\documentclass{amsart}

\usepackage{amssymb}
\usepackage{hyperref}
\usepackage{amsthm}
\usepackage{amsmath}
\usepackage{multicol}
\usepackage[all]{xy}
\usepackage{multirow}
\usepackage{float}
\usepackage{subfig}
\usepackage[pdftex]{graphicx}

\newtheorem{definition}{Definition}
\newtheorem{construction}{Construction}
\newtheorem{proposition}{Proposition}

\newcommand{\antip}[1]{#1^\star}
\newcommand{\spf}{{\bf F}}
\newcommand{\spb}{{\bf B}}
\newcommand{\spl}{{\bf L}}
\newcommand{\spr}{{\bf R}}
\newcommand{\spu}{{\bf U}}
\newcommand{\spd}{{\bf D}}
\newcommand{\ppf}{F}
\newcommand{\ppb}{B}
\newcommand{\ppl}{L}
\newcommand{\ppr}{R}
\newcommand{\ppu}{U}
\newcommand{\ppd}{D}
\newcommand{\spp}[1]{{\bf #1}}
\newcommand{\orthoback}[1]{{{#1}^b}}

\begin{document}

\title[A Construction of the Total Spherical Perspective]{A Construction of \break the Total Spherical Perspective \break in Ruler, Compass and Nail}

\author{Ant\'onio Ara\'ujo}
\address{  DCeT, Universidade Aberta, R. Escola Polit\'{e}cnica, 141-147, 1269-001 Lisboa, \newline Portugal.
         \newline
  CMAF-CIO,  Faculdade de Ci\^encias da Universidade de Lisboa, 1749-016 Lisboa, \newline
Portugal.
          \newline
  CIAC, Campus Gambelas, FCHS, Faro, Portugal.}
\email{antarj@gmail.com}

\begin{abstract}
We obtain a construction of the total spherical perspective with ruler, compass, and nail.
This is a generalization of the spherical perspective of Barre and Flocon to a 360 degree field of view. Since the 1960s, several generalizations of this perspective have been proposed, but they were either works of a computational nature, inadequate for drawing with simple instruments, or lacked a general method for solving all vanishing points.
We establish a general setup for anamorphosis and central perspective, define the total spherical perspective within this framework, study its topology, and show how to solve it with simple instruments. We consider its uses both in freehand drawing and in computer visualization, and its relation with the problem of reflection on a sphere.
\end{abstract}

\maketitle

\section{Introduction - Previous works on Wide Angle Perspectives}

It is often useful in the visual arts to depict a scene composed within a very wide angle of view. In their 1968 work \cite{BarreFlocon}, Barre and Flocon described a spherical perspective and provided a method to solve it with ruler and compass. Spherical perspective is however a misnomer of sorts, as only a hemisphere is projected, thus obtaining a 180 degree view around an axis. More recently, several works of a computational nature  have proposed various types of wide angle perspectives,  by generalizing the shapes of the projection surfaces or by expanding the angle of view up to 360 degrees (\cite{CORREIA}, \cite{UPTO360}). These works, however, require the use of computers, and are not amenable to the artistic practice of drawing freehand or with simple instruments such as ruler and compass.

The classic artistic device for representing views wider than 180 degrees is that of drawing sphere reflections from observation. In \cite{REFLECTIONS} sphere reflections were proposed for this reason as a wide angle perspective. Again this is a work of a computational nature, since sphere reflections are hard to solve. We will consider these difficulties ahead, and relate sphere reflections to spherical perspective.
 
Apart from computer scientists, several artists have tried their hand at generalizing spherical perspective:
 Dick Thermes is well-known for his paintings on spherical surfaces. These provide a full 360 degree view, but this is anamorphosis rather than perspective. His approach \cite{TERMES} to perspective  is based on gridding that follows from \cite{BarreFlocon}.

In a 1983 paper \cite{casas} F. Casas attempted the construction of a flat-sphere perspective, but with a misunderstanding of the geometry near the \em blowup\em, mistaking the lack of an isometry for that of a mathematically well-defined flattening. Not only does a flattening exist, there is an infinite number of them (two are described in the present work), and one must be chosen to specify a perspective. This Casas does not do, hence he is limited to a discussion of qualitative properties that apply to a class of total spherical perspectives. In \cite{moose} M. Moose implements the program of Casas through an ad-hoc gridding scheme that specifies an actual perspective, though not the one that generalizes \cite{BarreFlocon}.

There seems to be in these latter works somewhat of a misconception of the problem.  Barre and Flocon's contribution was not to create a flattening of the hemisphere. They chose one well known to cartographers, as acknowledged in \cite{Bouligand}. Their contribution was to fit it into the framework of a perspective, to provide a classification of all vanishing points and a way to solve them by elementary means. That is also what is provided in the present work for the total spherical perspective that generalizes Barre and Flocon's.

The reason why Barre and Flocon stopped at 180 degrees is a matter for the historian. This matter is discussed elsewhere  \cite{convocarte} at length, but briefly: it is not likely that they ignored that their flattening could be extended to 360 degrees, at least after their work with Bouligand \cite{Bouligand} in which the cartographic options were surveyed. My impression is that they stopped at 180 degrees for two reasons: their stated purpose of keeping linear deformations within reasonable bounds, and the difficulty of plotting line projections beyond the equator, where they stop being well approximated by arcs of circle. We shall see ahead how lines can be projected in a simple way beyond the 180 degree mark.

Previous to these publications, during the 1970s, the Belgian architect G\'{e}rard Michel made several experimental works in spherical perspective, drawing intuitively from urban scenes, but these drawings, along with brief hints on the artist's process, where published only recently \cite{Michel}. This rather discreet publication is the most important direct precursor of the present paper\footnote{It was not an influence on this work, however, as I only became aware of G\'{e}rard's work through private communication with the author - a fellow urban sketcher - while circulating the first draft of the current paper.} as Michel's drawings implicitly anticipate several of the formal results developed here. It is however an informal work with a different scope. As far as I know the present paper is the first systematic presentation of the total spherical perspective that clearly formalizes and solves it, that is, provides a classification of all lines and vanishing points, and a systematic method to find and project them, in a way amenable to drawing with simple instruments, from either orthographic plans or direct observation from nature. The coupling of the analytic and geometric formulations here presented also sugests an efficient method for computer rendering, but that is beyond the scope of the present paper and will be presented elsewhere.

\section{Perspectives}

Perspectives are representations of spatial scenes on a plane, with relation to an observer. Because visual occlusion is radial, most of the perspectives used by artists (classical, cylindrical, spherical) are central perspectives.
In what follows we shall define a central perspective as a composition of two maps: an anamorphosis followed by a flattening.

In dictionaries and perspective manuals the term anamorphosis describes an \em inverse \em problem that relates to its etymology ("to form again"): the game of finding the correct point to observe a picture. But it is more enlightening from a didactic and conceptual viewpoint (and also most in accordance with its role in the history of perspective) to define anamorphosis as a \em direct \em geometric construction that sets the foundation for  building a perspective. Vanishing points will be defined at the level of anamorphosis, even before one settles on a specific perspective.

\subsection{Anamorphosis, Topology and Vanishing points}

We shall speak of an observer to mean a point $O$ in three-dimensional euclidean space. We shall speak of a \em scene \em to mean a closed set in that space. 

A fundamental fact about vision is that, with few and notable exceptions, occlusion is radial. That is, points along the same ray from the viewer are seen as equivalent. Consequently, the draughtsman, like the astronomer, deals with rays rather than points and solid angles rather than lengths. This allows for a piece of trompe l'oeil to be created by the process of conic anamorphosis\footnote{We shall from now on omit the term \em conic \em and speak simply of anamorphosis, as we shall use no other type.}:
 a two-dimensional picture on a surface $S$ that creates, for an observer at $O$, the visual illusion\footnote{That the problem of anamorphosis (creating a two-dimensional simulacrum of a spatial scene) is solved by the construction of the same name is not a demonstrable mathematical property but an empirical fact of optics and physiology. Conic anamorphosis will fail to solve the problem when, for instance, linear optics approximation is not valid.} of a spatial scene $\Sigma$.

Let $\mathcal{R}_O$ be the set of rays from $O$. Let $S^2_O$ be the unit sphere centered at $O$. The isomorphism $P \mapsto \overrightarrow{OP}$ endows $\mathcal{R}_O$ with the topology of the sphere. Hence we can speak of the topological closure of a set of rays from $O$. Let $cl(X)$ denote the closure of a set $X$.


Let $\Sigma$ be a scene. $\Sigma$ defines a cone of rays from $O$, $C_O(\Sigma)=\{\overrightarrow{OP}: P \in \Sigma \}$. We say that $C_O(\Sigma)$ is the \em cone of sight \em of $\Sigma$ from $O$.



We say that a surface $S$ is \em central \em relative to a point $O$ if any ray from $O$ intersects  $S$ at most once.
We say that $S$ is an anamorphic surface relative to $O$ if it is a compact central surface relative to $O$. 

\begin{definition}
Let $S$ be an anamorphic surface for $O$ and $\Sigma$ a scene. We say that $C_{O,S}(\Sigma)=cl(C_O(\Sigma)\cap S)$ is the anamorphosis of $\Sigma$ on $S$ relative to $O$. Let $\Lambda:\mathbb{R}^3\setminus\{O\}\rightarrow S$ be the map $P \mapsto \overrightarrow{OP}\cap S$. We call $\Lambda$ the anamorphism (or conic projection) onto $S$ relative to $O$. We use the same name for the corresponding map $\Lambda: \mathcal{R}_O \rightarrow S$.
\end{definition}

From the point of view of the topologist, the purpose of perspective is the compactification of a spatial scene.
A line in space is closed but not bounded. Its anamorphosis onto a compact surface will be bounded but generally not closed. To make it closed, hence compact, we must add to it its \em vanishing points\em. We will define the vanishing points of a scene in an intrinsic way that does not depend on the specific perspective under consideration but only on the point $O$.\footnote{We assumed a scene to be closed in order not to get \em false \em vanishing points from this definition. We could drop the restriction by setting $\mathcal{V}_O(\Sigma)=cl(C_O(cl(\Sigma)))\setminus C_O(cl(\Sigma))$  instead, but that would be uglier and gain us little.}

\begin{definition}
 We say that $\mathcal{V}_O(\Sigma)=cl(C_O(\Sigma))\setminus C_O(\Sigma)$ is the set of \em vanishing points \em of the scene $\Sigma$ relative to $O$.
\end{definition}

We say that $\mathcal{V}_O(\Sigma)\cap S$ is the set of vanishing points of $\Sigma$ in the anamorphosis $C_{O,S}(\Sigma)$. 
Hence, the anamorphosis of $\Sigma$ onto $S$ is the union of $\Lambda(\Sigma)$, the strict conic projection onto $S$,  with its vanishing points.
The following is easy to show:
\begin{proposition}\label{vanishing_points_prop}
Let $r$ be a line and $r_0$ its translation to $O$. Then the set of vanishing points of $r$ in $S$ is $r_0 \cap S$. Analogously, let $H$ be a plane and $H_0$ its translation to $O$. Then the vanishing set of $H$ in $S$ (called its vanishing line) is $H_0 \cap S$. Hence the anamorphosis of a line $AB$ onto $S$ is a subset of the vanishing line of the plane $AOB$.    
\end{proposition}

\subsection{Anamorphosis onto a sphere}

Anamorphosis onto a sphere is the simplest anamorphosis, due to the natural isomorphism between the rays of sight and points of the sphere and is therefore very symmetric: all lines will project equally up to rotation and all have exactly two vanishing points.
All the other common anamorphoses in artistic practice (plane, cylinder, hemisphere) result in a more complicated description of vanishing points and line projections.

\noindent Let's recall a few generalities about circles on spheres:

A \em great circle \em is a circle on a sphere, defined by the intersection of the sphere with a plane through the origin. 

Given a point P on the sphere we call \em antipode point \em of P to the its diametrically opposite point on the sphere, and we denote it by $\antip{P}$. 

Two non-antipodal points $P$ and $Q$ on the sphere define a unique great circle, the intersection of the sphere with the plane $POQ$. We call this the $PQ$ great circle.

Each point $P$ on the sphere defines a family of great circles, all crossing both $P$ and its antipode $\antip{P}$, that covers the sphere. We call these circles $P$-great circles or $P\antip{P}$-great circles and call $P$ and $\antip{P}$ the poles of the family. A \em meridian \em is one connected half of a great circle. We call $P$-meridian or $P\antip{P}$-meridian to a meridian whose endpoints are at $P$ and $\antip{P}$.


We can now construct the anamorphosis of a generic spatial line:

\noindent Let $l$ be a line, $O\not \in l$. There is a single plane $H$ through $O$ containing $l$.  This plane defines a great circle $C$ on the sphere. The cone of sight of $l$ is $C_O(l)=\{\overrightarrow{OP} :P \in l\}$, a half-plane contained in $H$ whose boundary is the line $l_O$, the translation of $l$ to the origin. 
$l_0$ is the union of two rays from $O$ none of which is a ray of sight of an actual point of $l$ but correspond to the limit of the directions of sight of an observer that follows $l$ in both directions. Hence the strict conic projection of $C_O(l)$ onto the sphere is a meridian $M \subset C$ with its two antipodal endpoints missing. These two points are the intersection of $l_O$ with the sphere, and are the vanishing points of the line. Taking the topological closure of M we get the anamorphosis of $l$ onto $S$, which is a full meridian, being the union of M with the vanishing points. 

In the degenerate case $O\in l$, $l$ projects onto two antipodal points, with no vanishing points.

Analogously, we obtain the anamorphic image of a generic plane:

Let $H$ be a plane, $O \not \in H$. The cone of sight  $C_O(H)$ is a half-space whose boundary is $H_0$, the plane through $O$ parallel to $H$. The boundary is not contained in the set of rays of sight of individual points of $H$. The strict conic projection onto the sphere will be a hemisphere missing its boundary great circle C. Taking the closure of the conic projection we get the anamorphosis of $H$, a full hemisphere containing the \em vanishing circle \em $C$.  

In the case $O \in H$, $H$ projects onto a great circle with no vanishing points.

\subsection{From Anamorphosis to Perspective}

Just as the globe provides the cartographer with an ideal isometric model of the Earth, so does conic anamorphosis provide the artist with a topologically compact two-dimensional optical simulacrum of a spatial scene. But often, both cartographer and artist are willing to abandon these ideal models for the convenience of working on a flat surface. In the process, the cartographer pays with the loss of isometry, and the artist with the loss of the optical properties of anamorphosis.

Going from anamorphosis to perspective - as in going from globe to chart - can be done in an infinite variety of ways. Intuitively, we would like to say a perspective is an anamorphosis onto a surface $S$ followed by a flattening of $S$ onto a plane. However, if we try to do that naively we find that usually (e.g. in cylindrical or spherical perspective) the flattening map $\pi$ will only be well defined on a dense open set of $S$. We can however ensure that the inverse of $\pi$ extends to a continuous map between compact sets. This we must do to preserve the essential role of compactification, that is, of vanishing points.

\begin{definition}\label{perspective_definition}
Let $\Lambda: \mathcal{R}_O \rightarrow S$ be an anamorphism. We say that $\pi: U\rightarrow \mathbb{R}^2$ is a \em flattening \em of $S$ if $U$ is an open dense subset of $S$, $\pi$ is an homeomorphism, and there is a continuous map  $\tilde{\pi}:cl(\pi(U))\rightarrow S$ such that $\tilde{\pi}|_{\pi(U)}=\pi^{-1}$. We say that $p=\pi \circ \Lambda$ is the \em perspective \em associated to the flattening $\pi$. Let $\tilde{p}=\Lambda^{-1}\circ\tilde{\pi}$.
Given a scene $\Sigma$, we say that $\tilde{p}^{-1}(\Sigma)$ is the strict perspective image of $\Sigma$, that $\tilde{p}^{-1}(\mathcal{V}_O(\Sigma))$ is the vanishing set of $\Sigma$, and that the perspective image of $\Sigma$ is the union of its strict perspective image with its vanishing set. 
\end{definition} 

We find that the fundamental maps are not so much $p$ and $\pi$ but $\tilde{\pi}$ and $\tilde{p}$. That is, functional arrows are exactly the reverse of the naive view when we consider the topology. We resist the temptation to do away with tradition altogether and will still call $p$ the perspective. 

Apart from these formalities, a perspective should follow two informal but crucial requirements: First, it should be evocative of the visual experience, i.e., preserve at least some aspects of the spatial illusion that the anamorphosis affords. 

Second, it must be \em solvable \em. By solving a perspective we mean finding and plotting the images of the basic idealized objects of perspective - points, lines and planes - out of which more complex scenes are approximated. It follows from proposition \ref{vanishing_points_prop} that the image of a line $AB$ is a subset of the vanishing set of plane $AOB$. Hence solving a perspective reduces to solving its vanishing points. Whether a perspective is solvable depends on what tools we allow to solve it. A perspective may be solvable by a computer but inadequate to the unaided human artist. In this work we are concerned with a perspective that can be solved with elementary means, such as ruler and compass.

Among the infinite flattenings available for each surface $S$, a dense set will preserve nothing of visual interest, or will be too hard to solve. Considering the classical examples of perspective we see that the flattenings are chosen in order to relate naturally to their anamorphic surface, and to satisfy our two requirements:
In classical perspective the anamorphic surface is already a plane, so the natural flattening is the identity map (modulo scaling). Straight lines are preserved. In cylindrical perspective the anamorphic surface is a cylinder, which is a developable surface, so it can be cut and unfolded isometrically. Spatial lines become ellipses by anamorphosis and sinusoidals upon flattening. These can be plotted in good approximation by ruler and compass. In (hemi)spherical perspective the anamorphosis turns lines into arcs of great circle. There is no isometric flattening of a sphere onto a plane (we are in the position of the cartographer) so Barre and Flocon chose a flattening that preserves lengths along a set of meridians and that, crucially, turns arcs of great circle into arcs of circle in good approximation, which allows for plotting using ruler and compass.

There is an interesting symmetry between spherical and plane perspective. In classical perspective the flattening is trivial but the anamorphosis is not. In spherical perspective the opposite is true. 
This is because in classical perspective the plane of the anamorphosis can be identified with the plane of the perspective, while in the spherical perspective the anamorphic sphere can be identified with the set of directions, so the flattening in the former case and the anamorphosis in the latter can be identified with the identity map. This gives classical perspective its special status: since the flattening is trivial, anamorphosis is preserved. So called "perspective deformation" is a misnomer, resulting from the failure of the observer to stand at point $O$. The distortion of linear measurements (the so-called "paradox" of Leonardo) is a necessary consequence of the preservation of solid angles from $O$, and a feature, not a bug, of an effective \em trompe l'oeil \em\cite{convocarte}.

\section{Total spherical perspective: Flattening a sphere.}

We will now define our total spherical perspective, within the general scheme outlined above. We need an anamorphosis followed by a flattening. The anamorphosis is fully defined by the choice of the surface and the place of the observer. We take for anamorphic surface the unit sphere $S^2$, with the observer $\spp{O}$ at its center. We have already described the properties of this anamorphosis with regards to the projection of lines and planes. We must now discuss the flattening map.


We start by defining an observer-centered reference frame. We consider a ray stemming from $\spp{O}$, representing a privileged direction of sight. We call it the \em central ray of sight \em and to its axis we call the \em central axis of sight \em. We place an orthonormal right-handed coordinate system $xyz$ in $\spp{O}$, such that the positive side of the $y$ axis coincides with the central ray of sight. For easy reference we name the points where the three axes cut the sphere: we call \textbf{F}ront to the intersection of the central ray of sight with the sphere and \textbf{B}ack to its antipode point; \textbf{R}ight to the point where the $x$ axis touches the sphere and \textbf{L}eft to it's antipode; \textbf{U}p to the point where the positive $z$ axis touches the sphere and \textbf{D}own to its antipode, and we represent these points by their initials written  in bold. 

From now on we will simplify notations with the following convention: a spatial point and its plane projection will be denoted by the same letter, the spatial point in bold font and the projection in italic font. Hence, $P=p(\spp{P})$ will be the perspective of a spatial pont $\spp{P}$. In  particular, $P=\pi(\spp{P})$ will be the flattening of a point on the sphere, so the perspective images of reference points $\spf,\spb,\spl,\spr,\spu,\spd$ will be $F,B,L,R,U,D$ respectively.

We call the $y=0$ plane (orthogonal to the central axis of sight) the \em observer's plane\em. The observer's plane intersects the sphere in a great circle we call the \em equator\em.  We call the $x=0$ plane the \em sagittal plane\em, and to $z=0$ we call the \em plane of the horizon\em. We call \em central \em meridians to the $\spf$-meridians.
We call the half-space $y>0$ the \em anterior \em half-space (representing everything in front of the observer) and to the half-space $y<0$ we call the \em posterior \em half-space (representing all that is behind the observer).

We will now construct a flattening of the sphere. This is a construction for the \em azimuthal equidistant projection\em, well known to cartographers and astronomers. A restriction of this map to a single hemisphere is used in \cite{BarreFlocon}. Our purpose here is to establish a derivation of this map that is adequate to our purposes and show that it fits within our definition of perspective. 

Intuitively, we picture it thus: we look at the sphere as the union of its central meridians,  which we think of as inextensible threads. We cut the threads free at $\spb$, and pull them straight along their tangents at $\spf$, flattening them onto the plane tangent to the sphere at $\spf$ (see fig. \ref{disc360}). The straightened threads radiate from $\spf$, forming a disc $D$ of radius $\pi$. We call the boundary circle  of the disc the \em blowup \em of $\spb$, as we see this point as having been blown-up into the set of rays of the tangent plane of the sphere at $\spb$, each ray corresponding to one of the meridians from which $\spb$ could be approached . We now formalize this construction:

Let $D=\{(x,z)\in \mathbb{R}^2:x^2+z^2<\pi \}$. Let $\pi: S^2\setminus\{\spb\}\rightarrow D$ be the homeomorphism such that 

$\mathcal{C}_0$) each central meridian maps onto a line segment.

$\mathcal{C}_1$) distances are preserved along each central meridian.

$\mathcal{C}_2$) angles between central meridians are preserved at $\spf$.

Extending $\pi^{-1}$ to the closure of its domain we obtain the continuous map between compacts, $\tilde{\pi}: cl(D)\rightarrow S^2$. By continuity, it verifies $\tilde{\pi}(P)=\spb$ for all $P$ on the blowup circle $cl(D)\setminus{D}$, and $p=\pi \circ \Lambda$ defines a perspective according to definition \ref{perspective_definition}. 

Condition $\mathcal{C}_1$ means that the map is an isometry for each $\spf$-meridian separately. Since distances measured along great circles of the sphere are proportional to angles from the center, this means that if $\spp{P},\spp{Q}$ are points on the same $\spf$-meridian and if $P,Q$ are their images, then $|PQ|=\angle \spp{P}\spp{O}\spp{Q}$ up to multiplication by a scale factor\footnote{For points on the images of these meridians we will freely abuse notation and write equalities between angles and linear measures such as $|\overline{XZ}|=|\overline{XY}|+180^\circ$ to mean that these equalities are valid modulo product by the adequate scale factors.}. Conditions $\mathcal{C}_0$ and $\mathcal{C}_1$ imply that $\spf$ will be mapped to the center of the disc with images of the $\spf$-meridians radiating from it as line segments.

Condition $\mathcal{C}_2$ means that the angles between these segments at $\ppf$ will be equal to the angles of the corresponding meridians at $\spf$. This ensures the central meridian images will be distributed radially preserving their tangents at $\spf$, that is, they will look as if orthogonally dropped onto the tangent plane of the sphere at $\spf$. 
We call \em longitude \em of an $\spf$-meridian to the angle at $\spf$ between its tangent and that of the $\spf$-meridian through $\spr$. By $\mathcal{C}_2$, The longitude of a meridian equals the angle between its image and the $\ppf\ppr$ measuring line.

$\mathcal{C}_1$ and $\mathcal{C}_2$ together imply that the images of the two meridians of each great circle through $\spf$ form a diameter of the perspective disc and that distances are preserved within each diameter. For this reason we call \em measuring lines \em to the diameters of the perspective disc.


We will call equator of the perspective disc to the perspective image of the sphere's equator. This is a circle, with half the radius of the disc, upon which lie the images of points $\spr$,$\spl$,$\spu$,$\spd$. It divides the perspective disc into two parts: an inner disc that is the flattening of the anterior hemisphere, and an outer ring, between the equator and the blowup, that is the flattening of the posterior hemisphere (See fig. \ref{disc360}).


\begin{figure}
\hspace*{-1.2cm} 
\includegraphics[height=5.5cm]{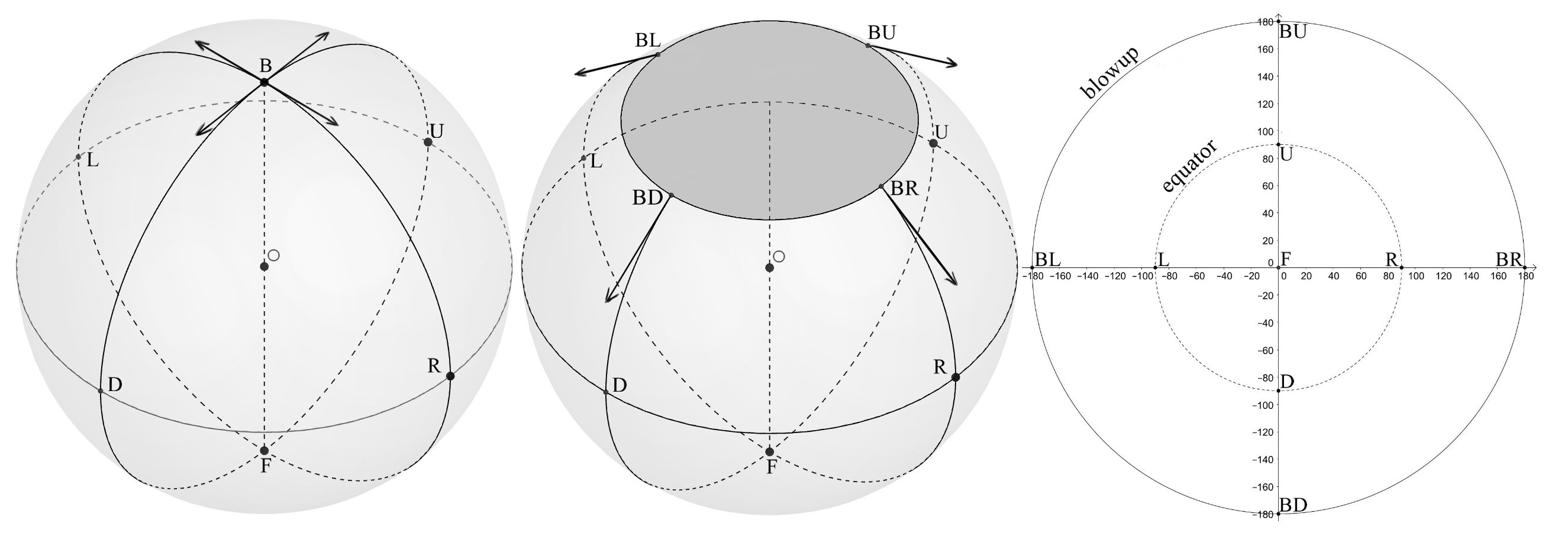}
\caption{Point $\spb$ is blown-up to a circle (\spb\spd-\spb\spr-\spb\spu-\spb\spl) and the punctured sphere is flattened onto the perspective disc. Distances (or angles from $O$) are preserved along points on each $\spf$-meridian. Here we see them marked along the $\spu$-$\spd$ and $\spl$-$\spr$  measuring lines, ranging from -180 to 180 degrees along each diameter of the perspective disc.}\label{disc360}
\end{figure}

In terms of the cartesian $(x,y,z)$ coordinates, the flattening composed with anamorphosis $\spp{P}\mapsto \overrightarrow{\spp{O}\spp{P}}/||\overrightarrow{\spp{O}\spp{P}}||$  gives the perspective map $f:\mathcal{R}_\spp{O}\setminus\{\overrightarrow{\spp{O}\spb}\}\mapsto D$,
\begin{equation}\label{exactxyz}
p(\left[x,y,z\right])=\frac{(x,z)}{\sqrt{x^2+z^2}}arccos\left(\dfrac{y}{\sqrt{x^2+y^2+z^2}}\right)
\end{equation}
this can be seen as projecting orthogonally against the $xz-$plane, taking the unit vector, and then scaling to a length equal to the value of the angle $\angle \spp{P}\spp{O}\spf$.

The natural set of spherical coordinates for this map is $(\rho, \lambda, \theta)$ with 
\begin{align}\label{natcoords}
\begin{split}
\rho&=|\spp{O}\spp{P}|=\sqrt{x^2+y^2+z^2}\\
\lambda&=\angle \spp{P}\spp{O}\spf=arccos\left(\dfrac{y}{|\spp{O}\spp{P}|}\right)\\
\theta&=arccos\left(\dfrac{x}{\sqrt{x^2+z^2}}\right)
\end{split}
\end{align}
where one can see $\lambda$ as the latitude, measured from $\spf$, and $\theta$ as the longitude, measured from $\spp{R}$. In these coordinates the anamorphosis becomes trivial, $p=\pi \circ \Lambda$ identifies with the flattening $\pi$ and we see clearly that the perspective image of $P$ doesn't depend on $\rho$, which was to be expected, since $\Lambda$ is a central projection:
\begin{equation}\label{trivial}
p(\rho, \lambda, \theta)=\pi(\rho, \theta)=\lambda(cos(\theta),sin(\theta))
\end{equation}

\begin{figure}
\centering
\includegraphics[height=10cm]{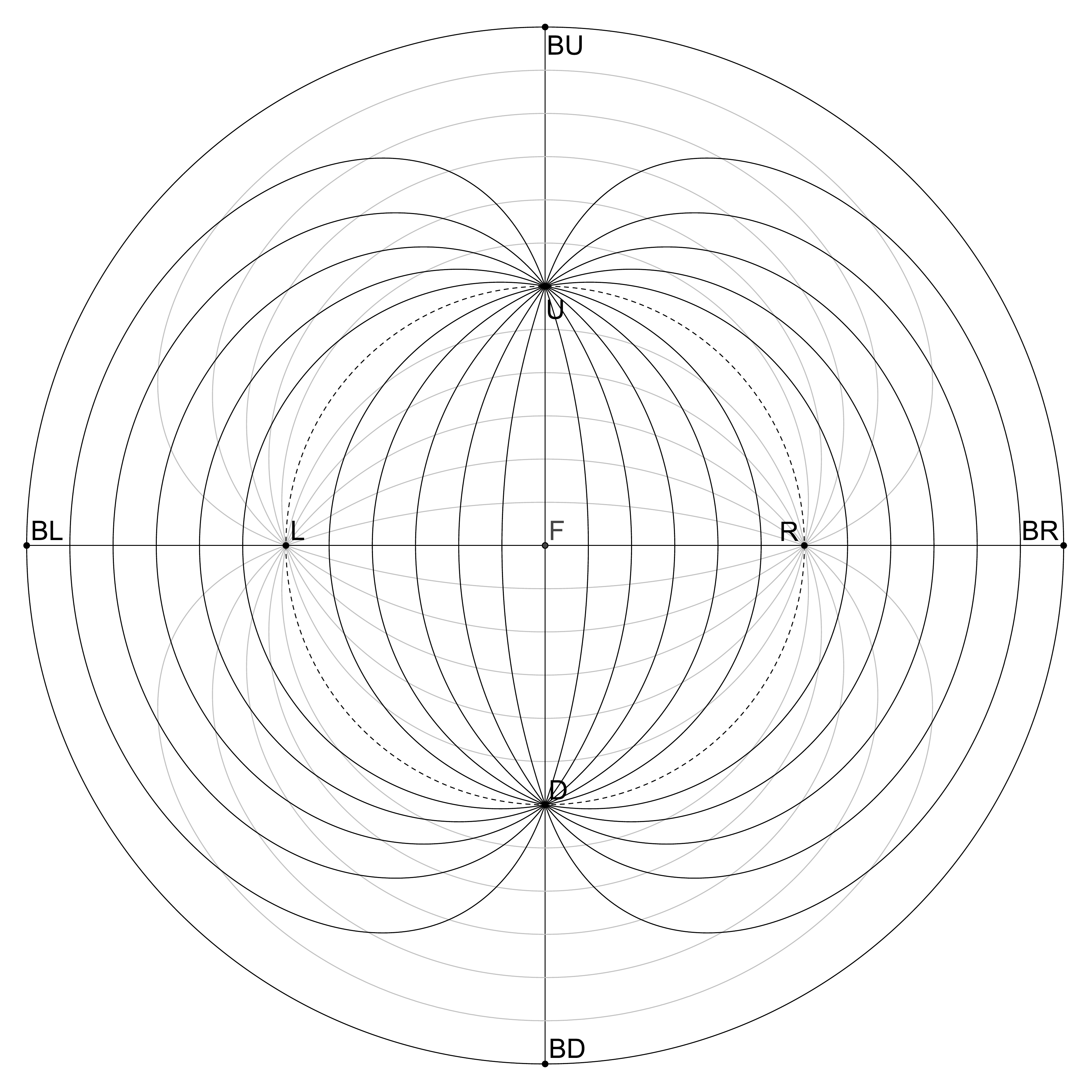}
\caption{Flattening of $\spu\spd$ (black) and $\spl\spr$ (grey) great circles from equation \ref{exactxyz}, at 15 degree intervals. These great circles correspond to the perspective images of vertical and frontal horizontal lines respectively.}\label{exact}
\end{figure}
\section{Solving a scene with ruler, compass, and nail}
The explicit form of the map obtained in the previous section would be enough for a pixel-by-pixel rendering of a scene on a computer. One could use it for a 3D engine or for displaying 360 degree photos captured with two 180 degree fisheye cameras pointing in opposite directions. It would however be of little use to the unaided human artist. 
A perspective that is useful for the draughtsman must stipulate how to solve a scene with  simple instruments. In what follows we will show how to solve a scene in total spherical perspective with ruler and compass, with allowance for marked rulers and plotting of arbitrary angles with protractors. The addition of a further tool - a nail - will further simplify practical constructions. We assume the data for the scene can be given either from direct measurements of angles (theodolite) or from Cartesian coordinates (architect's plan/orthographic views), and make stipulations for both cases.


A common technique to solve scenes in classical perspective is to make the plane of the perspective image do double or triple duty by superposing on it various orthogonal projections. This technique also works in spherical perspective. We will illustrate this in our first graphical construction:

\begin{construction}\label{on_equator}
\noindent Construction of the perspective image of a point on the observer's plane: Let $\spp{P}\neq \spp{O}$ be a point on the observer's plane. Then $\overrightarrow{\spp{O}\spp{P}}$ crosses the equator of the sphere. Hence the perspective image of $\spp{P}$ will be the point $P$  at the equator of the perspective disc such that $\angle P \ppf \ppr =\angle \spp{P}\spp{O}\spp{R}$. If the (x,y,z) coordinates of $\spp{P}$ are given, we can construct $P$ graphically thus: We make the plane of the drawing  represent both the perspective disc and the  back orthogonal projection view of the sphere onto the observer's plane, with $\ppf$ in the perspective disc coinciding with $\spp{O}$ in the orthogonal view and the disc scaled in such a way that the equator's perspective image coincides with its orthogonal projection image. Let the orthogonal image of $\spp{P}$ be $\orthoback{P}$. Plot $\orthoback{P}$ from its (x,z) coordinates. Then $\overrightarrow{\ppf \orthoback{P}}$  is the orthogonal projection of $\overrightarrow{\spp{O}\spp{P}}$ and $P$ is the intersection of  $\overrightarrow{\ppf \orthoback{P}}$   with the equator of the perspective disc. 
\end{construction}

The problem of solving a scene can be divided into two parts: plotting points and lines in the anterior half-space and in the posterior half-space. The anterior half-space is solved in \cite{BarreFlocon}. We will give here a very condensed version of that method, adapted to our needs.

\subsection{Solving the anterior hemisphere}

It is well known (\cite{Bouligand}, \cite{BarreFlocon}) that the perspective image of lines in the anterior half-space is well approximated by arcs of circle. This is important for two reasons: the first is that in drawing practice, arcs of circle are easy to trace with ruler and compass or even freehand; the second is that three points determine a unique arc of circle\footnote{Given three non-collinear points $P,Q,R$ on a plane, find the perpendicular bisectors of $\overline{PQ}$ and $\overline{QR}$ and intersect them to find the center of the circle. But the freehand draughtsman uses a different strategy: rather than attempt to find the center he eyeballs a line of constant curvature through $P,Q,R$. A draughtsman practices lines of constant curvature just as he does lines of zero curvature (straight lines).} so that few points have to be found.

We have to consider two cases: frontal and receding lines.

\subsubsection{Images of frontal lines.}

We say that a plane is \em frontal \em if it is parallel to the observer's plane. We say that a line is frontal if it lies on a frontal plane.
Let $l$ be a line on a frontal plane $H$. First suppose that $H$ is not the observer's plane.
Translating $l$ to $\spp{O}$ we find it has two vanishing points $\spp{V}$ and $\antip{\spp{V}}$ which define diametrically opposite points on the sphere's equator. Their images are found by drawing the translated line directly on the perspective disc, to obtain its intersection with the disc's equator (as in construction \ref{on_equator}). Next, we find a third point. If $l$ is not vertical, it intersects the sagittal plane at some point $\spp{P}$. We plot the measure of the angle  $\angle \spp{P}\spp{O}\spf$ on the vertical measuring line. If $l$ is vertical then it crosses the plane of the horizon and we measure instead the angle with the central axis at this point, and plot it on the horizontal measuring line. The image of $l$ is well approximated by the arc of circle $VP\antip{V}$ (see fig. \ref{anterior_lines}). If $\spp{P} \in \overrightarrow{\spp{O}\spf}$ then $P\equiv \ppf$, so $l$ projects onto a diameter of the disc.

Now suppose that $H$ is the observer's plane. We get $V$ and $\antip{V}$ as above, but $P$ will now project on the equator of the perspective disc. The arc of circle will be one half of the equator.

Note: The natural angles to measure with a theodolite when drawing from nature are those on the horizontal and vertical measuring lines - hence our focus on those measurements.

\begin{construction}\label{arbitrary_anterior}
Perspective of an arbitrary point $P$ on the anterior half-space. Consider the frontal plane going through $\spp{P}$ and on it a vertical line $v$ and a horizontal line $h$ going through $\spp{P}$. We already know how to solve these lines. The perspective image of $\spp{P}$ will be found at the intersection of the images of $v$ and $h$.
\end{construction}

\subsubsection{Images of receding lines.}

We say that a line is a \em receding line \em if it intersects the observer's plane at a single point.
Let $\spp{P}$ be the point of intersection of a receding line $l$ with the observer's plane. We plot P as in construction \ref{on_equator}. The plane $H$ defined by $\spp{O}$ and $l$ must also intersect the equator at the antipodal point  $\antip{P}$. To find a third point, we translate $l$ to $\spp{O}$ and intersect it with the sphere to find the two vanishing points. One of these will be on the anterior hemisphere, so we plot it by construction \ref{arbitrary_anterior}. Let its image be $V$. We trace the auxiliary arc of circle $PV\antip{P}$ that is the image of the plane $H$ in the anterior disc. The anterior image of $l$ will be the part of the arc that lies between $V$ and $P$.

If $l$ lies on a plane  through an $\spf$-meridian, it will project into a diameter of the disc. A particular case is that of the central lines. We say that a line is central if it is perpendicular to the observer's plane. In this case $V \equiv\ppf$, hence $V$ will be between $P$ and $\antip{P}$, the image of $H$ will be the straight line segment $P\antip{P}$ and the image of $l$ will be the segment $\overline{P\ppf}$  (see fig. \ref{anterior_lines}). Hence, central lines project as in classical perspective.

\begin{figure}
\centering
\includegraphics[height=10cm]{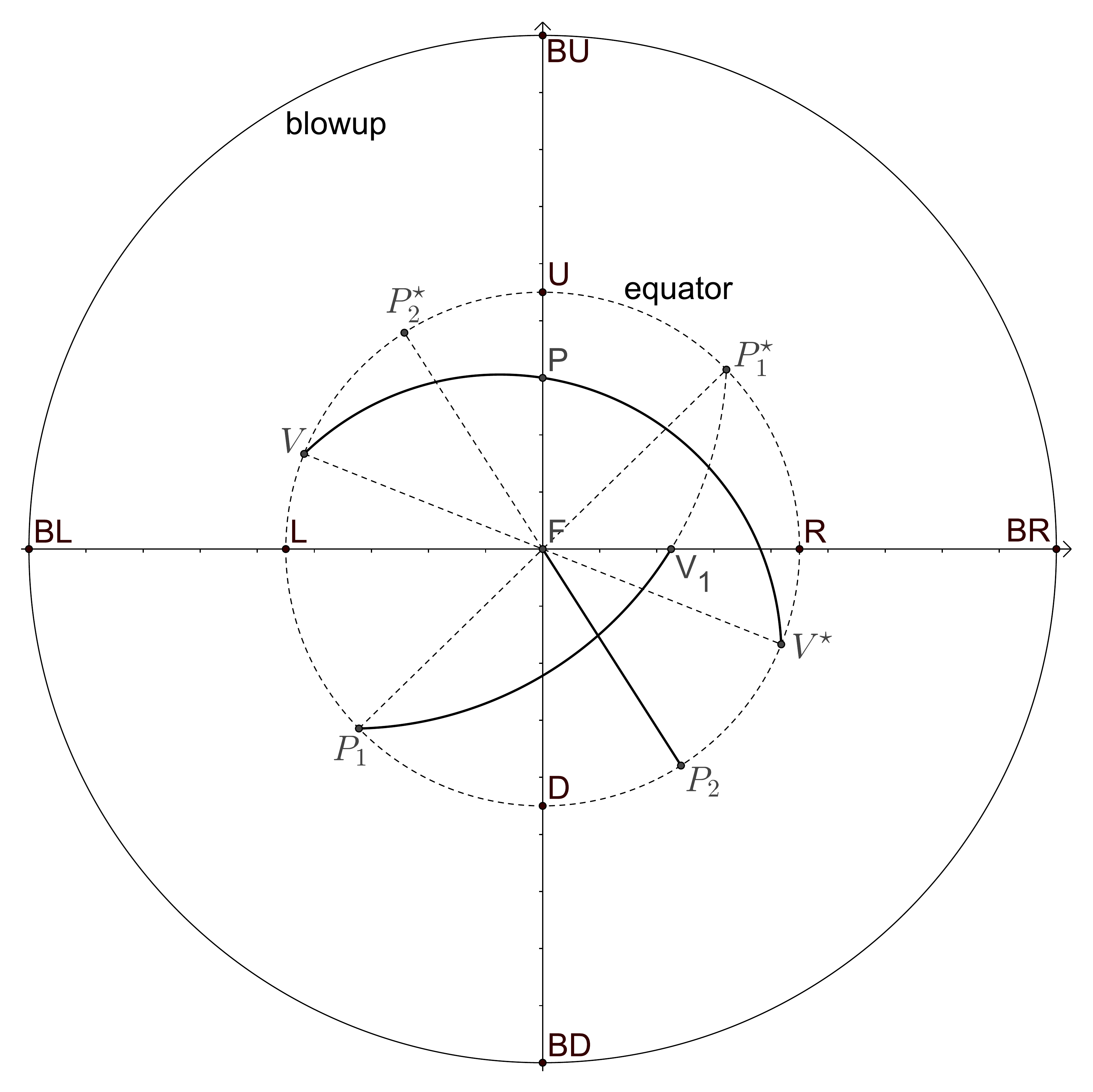}
\caption{Perspective images of lines restricted to the anterior hemisphere. Arc $VP\antip{V}$ is the image of a frontal line. Arc $P_1V_1\antip{P_1}$ is the image of the plane of a receding line, and the arc from $P_1$ to $V_1$ is the image of the line itself. The line segment $\overline{P_2\ppf}$ is the image of a central line. The image of its plane extends to the image of the antipode of $P_2$. }\label{anterior_lines}
\end{figure}

This ends our condensed review of (hemi)spherical perspective as presented in \cite{BarreFlocon}. Outside of the anterior   disc the images of lines are no longer well approximated by circles. This is one of the reasons why Barre and Flocon's perspective was limited to $180^\circ$. We will now show how to extend it to the full $360^\circ$ view.

\subsection{The full $360^\circ$}

We now wish to project the full image of a generic spatial line. We know lines project onto meridians. It is best to start by solving the complete great circle and then delimit the meridian by finding its end points. 
Our strategy is to piggyback on the known procedure for the anterior half space and use it to obtain a plot of the full great circle. The key lies in plotting antipodal points.
On what follows, let $r$ be the radius of the perspective disc (the size of the sphere and of the perspective disc are arbitrary up to a choice of scale factor).

\subsubsection{Plotting antipodal points}

\begin{proposition}\label{ruler}
Let $\spp{P}$ be a point in space such that $\spf \neq \spp{P}\neq \spp{O}$. Then $\antip{P}$ is the point on $\overrightarrow{P\ppf}$ such that $|P\antip{P}|=r$.
\end{proposition}

\begin{proof}
Let $C$ be the great circle through $\spf$ and $\spp{P}$. Since $\spf\in C$, the image of $C$ is a measuring line. $C$ contains the $P$-meridian $G=\spp{P}\spf\antip{\spp{P}}$. Since $\spb \not \in G$ and the flattening is continuous, the image of $C$ is connected and preserves the ordering of points, so $G$ projects to the segment $\overline{P\antip{P}}$ and $\antip{P}\in \overrightarrow{P\ppf}$. Since $G$ is on a measuring line and $\angle{\spp{P}\spp{O}\antip{\spp{P}}}=180^\circ$, then $|P\antip{P}|=r$.
\end{proof}

\,

Intuitively: on the sphere, in order to find $\spp{\antip{P}}$ from $\spp{P}$ we can follow the single $\spf$-great circle that crosses $\spp{P}$, for a length of 180 degrees along the direction that crosses $\spf$. But $\spf$-great circles flatten onto diameters along which lengths are preserved. So find $\antip{P}$ from $P$ by following the diameter for half its length along the $\overrightarrow{P\ppf}$ direction.

This proposition allows us to easily plot the antipode of an already plotted point $P$. Just draw line $\ppf P$, open the compass from $P$ with radius equal to the radius of the perspective disc, and intersect with $\ppf P$ to find $\antip{P}$. Or, if using a marked ruler, pass the ruler through $P$ and $\ppf$ with the zero mark at $P$, and plot $\antip{P}$ where the ruler marks $r$.

For the purposes of freehand drawing of a perspective it is often useful, when plotting points nearer to the equator than to $\ppf$, to use instead the following result:

\begin{proposition}\label{freehand}
Let $\spp{P}$ be a point in space such that $\spf \neq \spp{P}\neq \spp{O}$. Let $P_\ppb$ be the intersection of $\overrightarrow{P\ppf}$ with the blowup of $\spb$. Then $\antip{P}$ is the point on $\overrightarrow{P\ppf}$ such that $\left|\overline{\antip{P}P_\ppb}\right|=\left|\overline{\ppf P}\right|$. Also, $\left|\overline{\antip{P}\ppf}\right|=\left|\overline{P(-P_\ppb)}\right|$, where $-P_\ppb$ is the point on the perspective disc diametrically opposite to $P_\ppb$.
\end{proposition}

\begin{proof}
The plane $H=\spf \spp{O}\spp{P}$ defines a great circle $C$ that contains $\spp{P}, \spp{\antip{P}}, \spf$, and $\spb$. On that plane, the lines $\spp{P}\spp{\antip{P}}$ and $\spf \spb$ intersect at $\spp{O}$, and therefore we have the equalities between opposing angles $\angle{\spp{P}\spp{O}\spf}=\angle{\antip{\spp{P}}\spp{O}\spb}$ and $\angle{\spp{P}\spp{O}\spb}=\angle{\antip{\spp{P}}\spp{O}\spf}$. Since $C$ is a great circle through $\spf$, $\tilde{\pi}^{-1}(C)$ is a diameter of the perspective disc. On $C$ we have a cyclic order of points $\spp{P}-\spf-\antip{\spp{P}}-\spb$. Since $\tilde{\pi}$ is continuous, the order is preserved on the perspective image and we have $(-P_\ppb)-P-\ppf-\antip{P}-P_\ppb$ where $P_\ppb$ and $(-P_\ppb)$ are the points of the blowup corresponding to the directions of the two meridians of $C$ at $\spb$. Because distances are preserved along measuring lines, the two angle equalities above imply $\left|\overline{P\ppf}\right|=\left|\overline{\antip{P}P_\ppb}\right|$ and  $\left|\overline{P(-P_\ppb)}\right|=\left|\overline{\antip{P}\ppf}\right|$ respectively.
\end{proof}

The practical interest of this proposition lies in the fact that for freehand plotting of lines in the full spherical perspective it is often easier to transport the measurement $\left|\overline{P\ppf}\right|$ by eye than to transport the radius of the disc without an actual compass or ruler. But, having a compass at hand, or a marked ruler, the use of proposition \ref{ruler} makes for very efficient plotting of antipodes.

Now we can plot the image of a great circle's posterior meridian from the image of its anterior meridian:

\begin{construction}\label{makefatline}
Construction of fat lines: Let $C_a$ be the perspective image of the anterior meridian of a great circle $C$ on the sphere. To obtain an approximation of the posterior image $C_p$ of $C$, trace an arbitrary number of measuring lines $m_1, \ldots, m_K$ through $\spf$. Intersect each of these lines  with $C_a$ to get  points $Y_1, \ldots, Y_K$, and use proposition \ref{ruler} to
to obtain the antipodes $\antip{Y}_i$. Through each successive three of these points we trace an arc of circle, thus getting overlapping arcs $\antip{Y}_1\antip{Y}_2\antip{Y}_3$, $\antip{Y}_2\antip{Y}_3\antip{Y}_4$, etc. These overlapping arcs form a \em ``fat line'' \em that approximates $C_p$. The degree to which successive arcs fail to exactly overlap (how ``fat'' the envelope of these arcs is) indicates the amount of error in the approximation and the need to increase the number of measuring lines $m_i$\footnote{Alternatively just plot successive non-overlapping arcs and judge the error by how much the tangents differ at the endpoints of successive arcs.}(see fig. \ref{fatlinesexample}).
\end{construction}
A trained draughtsman can easily eyeball a line of constant curvature between three well spaced points, so that for freehand drawing the arcs of circle in this construction can be quickly obtained without actually finding their centers. The following construction is a very efficient way to obtain quickly as many points as necessary for the interpolation:
\begin{construction} Ruler, Compass and Nail: The practical draughtsman, being given $C_a$, and armed with a marked ruler, will simplify construction \ref{makefatline} as follows. Suppose the ruler has a zero mark and an $r$ (radius of the disk) mark. Stick a nail on the center of the perspective disc, and sliding the ruler against the nail to ensure that it always touches $\ppf$, make its zero mark slide along the curve $C_a$. Then the $r$ mark will automatically slide along the antipodal curve $C_p$, and one can easily plot a great number of antipodal points very quickly, allowing $C_p$ to be interpolated by hand with good precision by joining each set of three successive points with arcs of constant curvature (see fig. \ref{fatlinesexample}).
\end{construction}

 It is easy to imagine a simple mechanical device to make this construction even more efficient: a ruler with a slit along its length, the length of the slit defined by a sliding stopper, so that in each drawing it would be fixed to the radius of the perspective disk. On one end of the slit there would be a spotter and on the other end a pencil point. As the user follows half of a meridian with the spotting end, the nail slides along the slit and the pencil end automatically traces the antipodal meridian in a continuous line, with no interpolation needed. A further adaptation of this device would allow the compass tracing the anterior meridian to guide the pencil end, tracing also the posterior meridian in the same motion.

With practice none such refinements are needed. Even the nail can remain merely conceptual, although a physical one can make quite a difference in efficiency (do try it with a thumbtack!). 
  
We are now ready to plot arbitrary lines in full perspective. 
We have the following cases:

\subsubsection{Images of frontal posterior lines.}

Let $l$ be a line in a frontal posterior plane. Let $H$ be the plane defined by $l$ and $\spp{O}$, and $C$ its great circle. Suppose $l$ is not vertical. Then $l$ crosses the sagital plane at a point $\spp{P}$, and $P$ will be a point on the posterior ring of the perspective disc, such that $|\ppf P|=\angle \spp{P}\spp{O}\spf$, on $\overrightarrow{\ppf \ppu}$ or on $\overrightarrow{\ppf \ppd}$ according to whether $\spp{P}$ is above or below the observer. By proposition \ref{ruler}, the antipode of $\spp{P}$ will map to the point $\antip{P}\in \overrightarrow{P\ppf}$ such that $|P\antip{P}|=r$. This point will be in the anterior perspective disc, therefore we can approximate the anterior image of $C$ by the arc of circle $C_a=V\antip{P}\antip{V}$, where $V$ and $\antip{V}$  are two vanishing points at the equator.  We can now use construction \ref{makefatline} to obtain the fat line approximation of the antipodal image $C_p$. Then the full image of $C$ will be $C_a\cup C_p$ and the image of the $l$ will be the $C_p$ meridian. (see fig. \ref{posteriorlinesfig}). Note that if $\spp{P} \equiv \spb$ then $l$ flattens to two disconnected line segments: a diameter of the full disc minus its intersection with the inner disc. This line is however connected when considered in the  topology induced by $\tilde{\pi}$, since the blowup of $\spb$ - seen as a single point - connects both segments.

\begin{construction}
We can now plot an arbitrary point $\spp{P}$ on the posterior half-space: pass vertical and horizontal lines through $\spp{P}$, plot them according to the procedure just described, and intersect their images to find $P$.
\end{construction}

\subsubsection{ Images of receding lines. }

Let $l$ be a line that crosses the observer's plane at a single point $\spp{P}$. Let $H$ be the plane defined by $l$ and $\spp{O}$, and $C$ its great circle.
By construction \ref{on_equator} and proposition \ref{ruler} we obtain the points $P$ and $\antip{P}$ on the perspective disc's equator.
Displacing $l$ to the origin we obtain two vanishing points; one on the anterior hemisphere, Let it be $\spp{V}$, and its antipode $\antip{\spp{V}}$ on the posterior hemisphere. Plot $V$ by construction \ref{arbitrary_anterior}, then use proposition $\ref{ruler}$ to plot $\antip{V}$.  The arc of circle $C_a=PV\antip{P}$ is the anterior image of $C$. From that $C_a$ plot the antipodal meridian $C_p$ by construction \ref{makefatline}. This plots the full image of the great circle $C$. To get the image of $l$, discard the arc $V\antip{P}\antip{V}$ (see fig. \ref{posteriorlinesfig}).

If $l$ is on the plane of an $\spf$-meridian, it will project into a measuring line. 
In the particular case in which  $l$ is a central line, then $\spp{V}\equiv \spf$ and $\antip{V}\equiv \spb$, 
and the image of $l$ will be a radius of the perspective disc. The intersection point of $\overrightarrow{\ppf P}$ with the blowup circle codifies both the vanishing point $\spb$ itself and the direction (or the meridian) from which it is approached as the line of sight follows $l$ to infinity. 

%
%
\begin{figure}
\includegraphics[height=12cm]{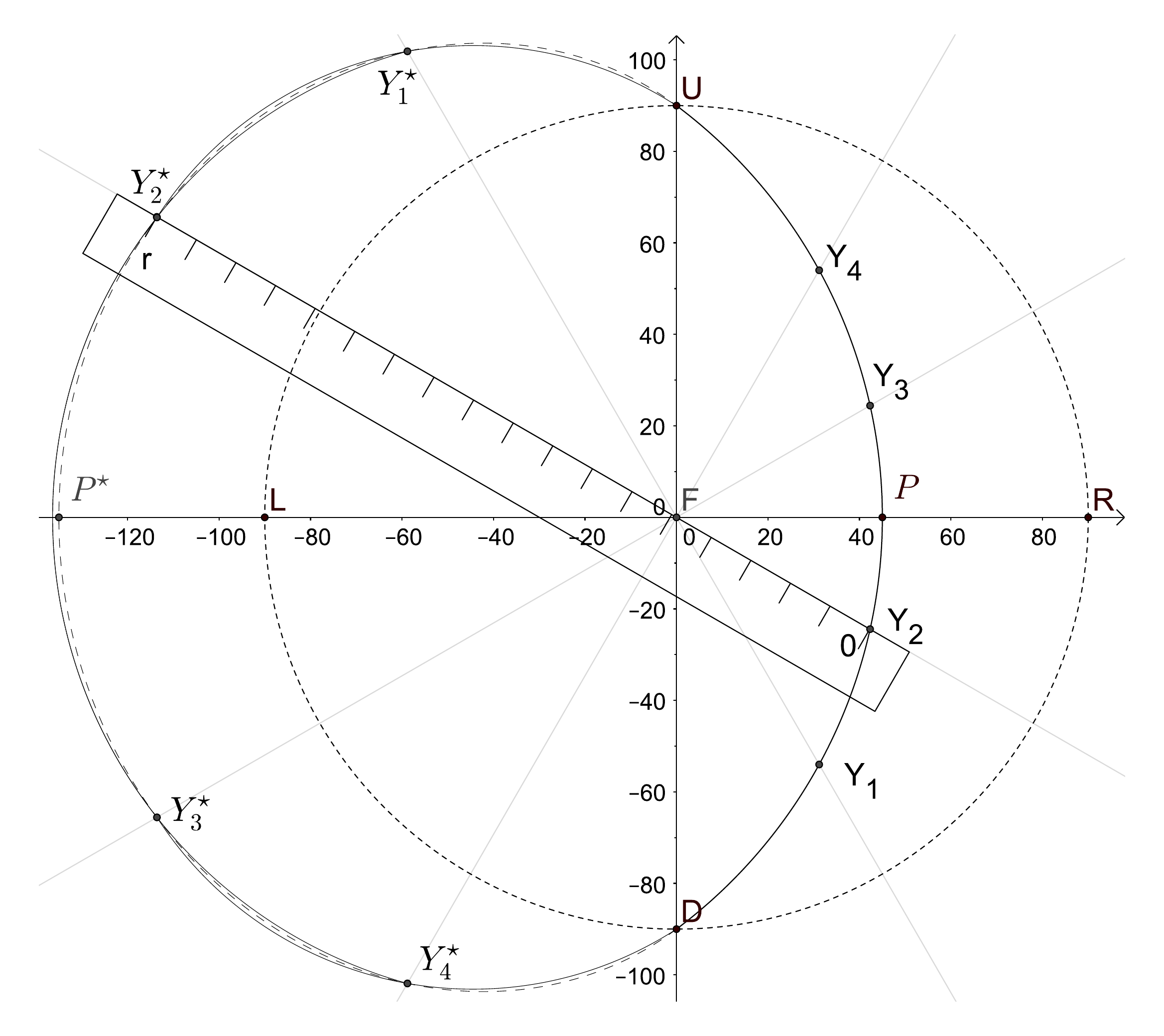}
\caption{Perspective image of a $\spu \spd$ great circle placed at 45 degrees to the observer's right. The dashed line $\ppu\antip{P}\ppd$ is the exact antipodal line of the arc of circle $\ppu P \ppd$ on the anterior view. The filled lines are a 'fat line' approximation to this line, obtained by interpolating through the antipodes of the four points $Y_i$, obtained by intersection with measuring lines set at 30 and 60 degrees to the horizontal axis. Even this coarse approximation fails at its worst by little more than one degree, but a great number of points could be obtained quickly through ruler and nail: stick a nail at point \ppf; then, as you lead the 0 mark of the ruler over the anterior curve $\ppu P \ppd$, sliding the ruler along the nail, the r mark will automatically trace the posterior curve $\ppu \antip{P} \ppd$.}
\label{fatlinesexample}
\end{figure}
\begin{figure}
\includegraphics[height=11cm]{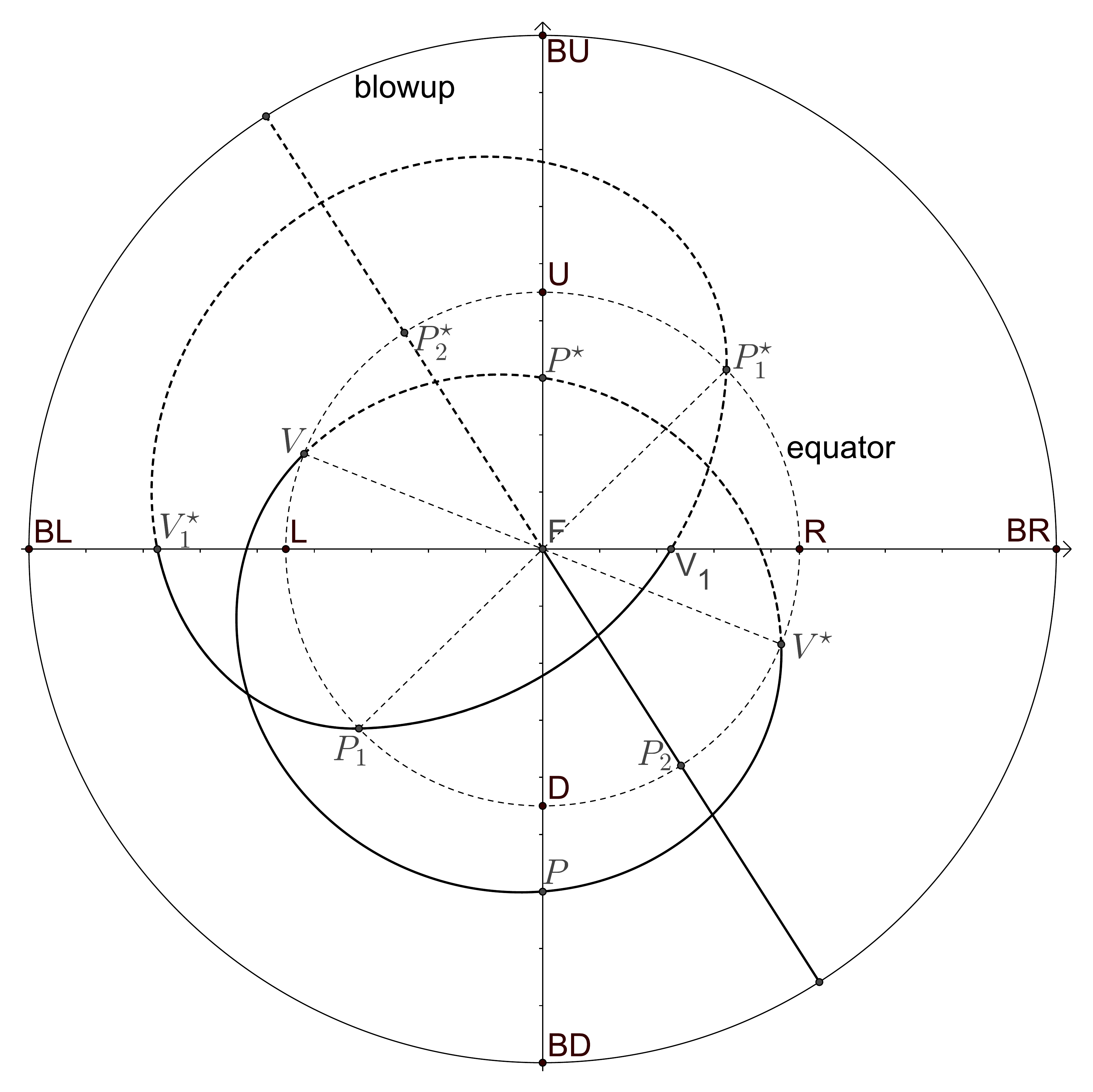}
\caption{Total spherical perspective of frontal, receding, and central lines. Arc $VP\antip{V}$ is the image of a frontal posterior line. Arc $V_1P_1\antip{V_1}$ is the image of a receding line. The radius through $P_2$
is the image of a central line. Note in each case the dashed curve that extend the image of each line to the image of its great circle. }
\label{posteriorlinesfig}
\end{figure}
%
%
%
%
%
%
%

%
%
%
\subsection{plotting curves of constant angular elevation}
What we have just learned is enough to solve a scene when we have the Cartesian coordinates of its points - for instance when drawing from an architectural plan. When drawing from observation, however, the artist  measures only the angles subtended by objects. We have already seen what the natural spherical coordinates are for this perspective (the angles $\lambda$ and $\theta$ defined above), and it is possible to construct a simple device to measure these angles directly, but the more habitual set of angles are the horizontal angle $\xi$ together with the angular elevation $\zeta$, defined thus: $\xi$ is the angle between the central ray $\overrightarrow{\spp{O}\spf}$ and the orthogonal projection of $\overrightarrow{\spp{O}\spp{P}}$ against the plane of the horizon. $\zeta$ is the angle between  $\overrightarrow{\spp{O}\spp{P}}$ and its orthogonal projection on the plane of the horizon. These are the angles one would measure with a standard theodolite.

Lines of constant horizontal angle $\xi$ are the images of vertical lines and we already know how to plot them. 
Lines of constant angular elevation are circles on the anamorphic sphere obtained by intersection with horizontal planes. For short we will call these circles and their images \em parallels\em\footnote{This is mixing metaphors, since in keeping with the geographical analogy, parallels should be the circles of the planes parallel to the equator, not those orthogonal to the $\spu\spd$ axis, but it's a convenient term, and we will use it with apologies.}

In the anterior hemisphere we  approximate parallels by arcs of circles in the manner of \cite{BarreFlocon}:
\noindent Let $h$ be a parallel of constant angular altitude $\zeta$. $h$ intersects the sphere's equator at two points $\spp{P}_\spl$ and $\spp{P}_\spr$ on the left and right side of the sagital plane respectively and intersects the anterior sagital plane at a point $\spp{P}$. Then $P_\ppl$ and $P_\ppr$ will be at the disc's equator and $\angle P_\ppr \ppf \ppr=\angle P_\ppl \ppf \ppl=\zeta$, and $P$ will be at the vertical segment $\overline{\ppu \ppd}$, and $|\ppf P|=\angle \spf \spp{O}\spp{P}=\zeta$. We take the arc of circle $P_\ppr P P_\ppl$ as the approximation to the anterior image of the parallel $h$.
To plot the posterior part of the parallel we make use of the following proposition:

\begin{proposition}\label{plotparallels}
Let $h$ be a parallel on the anamorphic sphere. Let $\spp{P}\neq \spf$ be a point of $h$. Let $M=\overrightarrow{\ppf P}\cap \varepsilon$ where $\varepsilon$ is the equator of the perspective disc. Let $Q$ be the point such that $M$ is the midpoint of $\overline{PQ}$. Then $Q$ is the perspective image of a point of $h$.  
\end{proposition}

\begin{proof} 
Parallels and $\spf$-meridians are invariant by reflection across the observer's plane (because so are their defining planes and the sphere itself and hence their intersection). Then the intersection of a parallel and an $\spf$-meridian is also invariant for reflection across the observer's plane, and since it is an intersection of circles, it is made up of a whole circle, or of zero, one, or two mirror symmetric points. 
Let $m\subset \overrightarrow{\ppf P}$ be the radius through $P$. $m$ is the image of the $\spf$-meridian $C$ that crosses $\spp{P}$. Hence $\spp{M}=\tilde{\pi}(M)$ is the point where $C$ crosses the sphere's equator.  Since $|PM|=|MQ|$ and $m$  is a measuring line, then $\angle \spp{P}\spp{O}\spp{M}=\angle \spp{M}\spp{O}\spp{Q}$, and since $P$ and $Q$ lie on the plane of $C$, orthogonal to the observer's plane, then $P$ and $Q$ are mirror symmetric relative to the observer's plane, hence $Q$ is on $h$.
\end{proof}
\begin{construction}
To plot the posterior half of a parallel $h$, plot first the anterior half $h_a$ as an arc of circle, then plot a set of measuring lines $r_i$, intersect them with $h_a$ at points $Y_i$, find the antipodal points $\antip{Y_i}$ from proposition \ref{plotparallels}, and trace a fat line through the $\antip{Y_i}$.
\end{construction}
Figure \ref{elevationfig}.a) shows a computer plot of parallels and verticals calculated directly from map \ref{exactxyz}. Figure \ref{elevationfig}.b) shows the approximation of the parallels of elevation 10, 45, 80, and 85 degrees plotted by proposition \ref{plotparallels} applied to the inner disc approximation. We see that the curves are not smooth at the equator, this being more noticeable when closer to $\spu$. This is an artefact of the approximations, as we see from equation \ref{exactxyz} that the perspective images of constant elevation curves are differentiable. The error stems not from proposition \ref{plotparallels}, which is exact, but from the initial approximation of the parallel by an arc of circle inside the anterior disc. Near the equator one should favour the method of the previous section instead. The practical draughtsman will however just smooth the edges at the equator and use parallels whenever convenient.

\begin{figure}
   \centering
    \subfloat[Grid of verticals and parallels, plotted directly from function \ref{exactxyz}. ]{{\includegraphics[width=12cm]{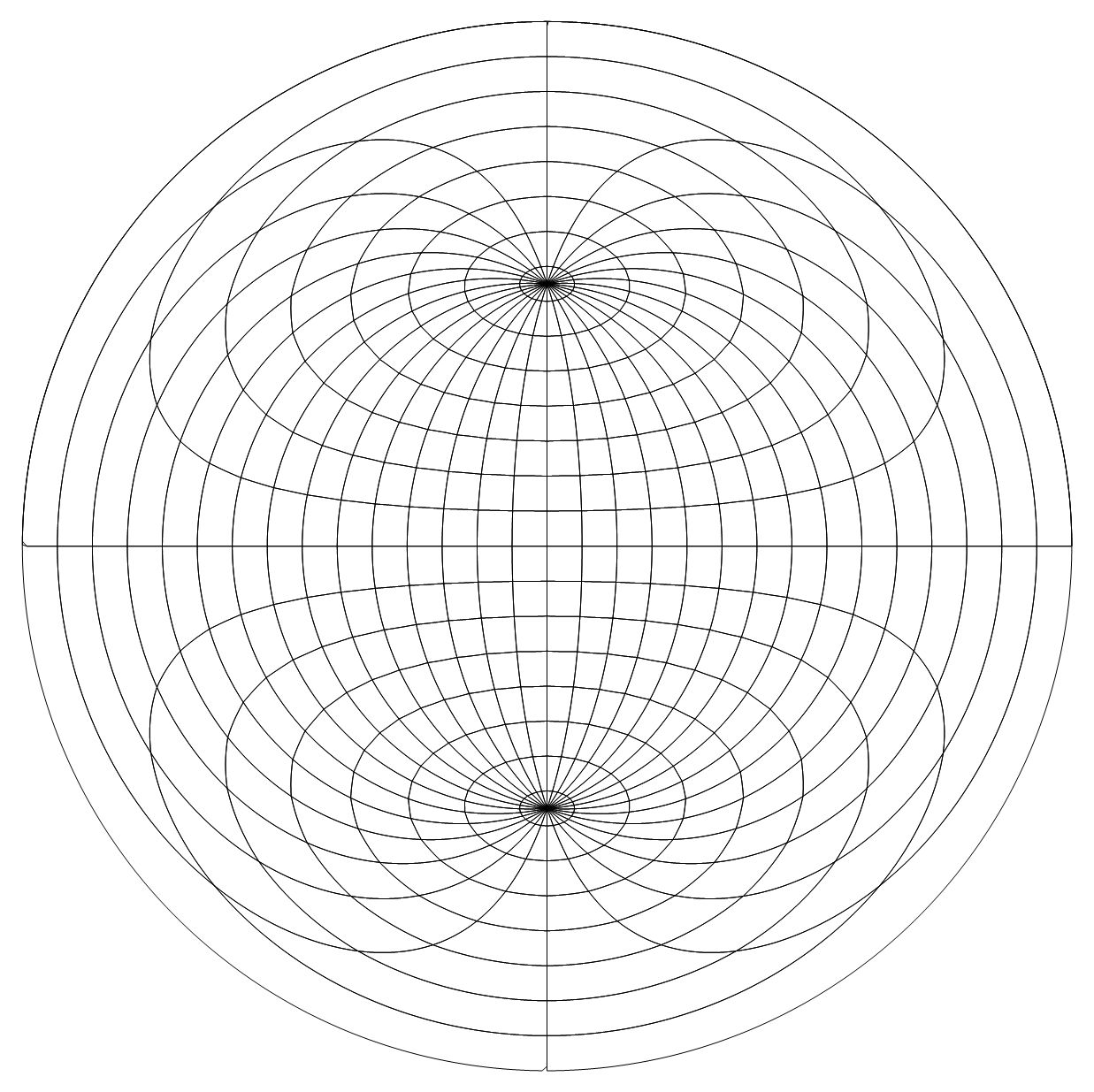} }}
   \qquad
   \subfloat[Ruler and compass approximation of curves of constant angular elevation of 10,45,80, and 85 degrees. Note the break of differentiability at the equator.]{{\includegraphics[width=12cm]{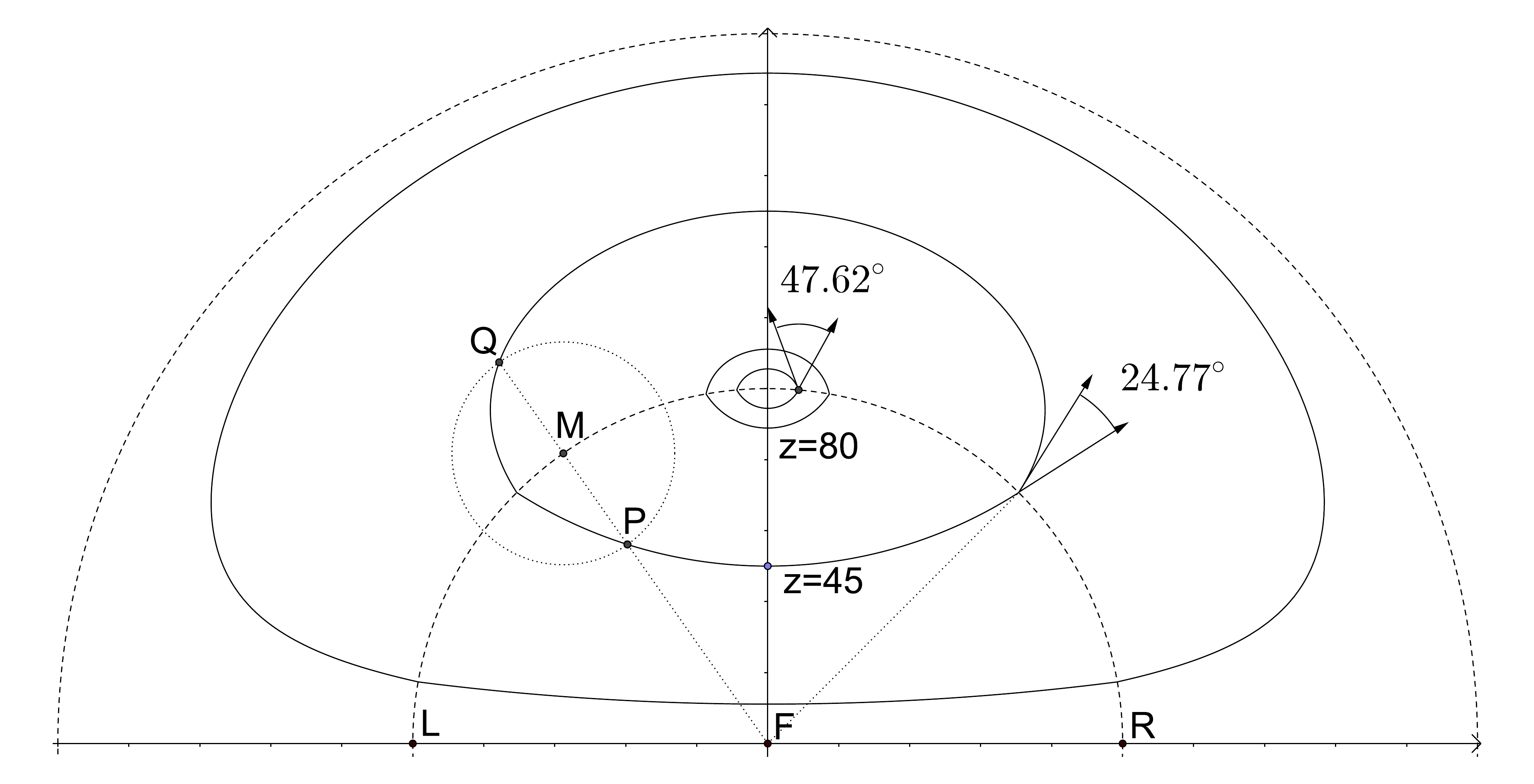} }}%
    \caption{Lines of constant angular elevation}%
   \label{elevationfig}
\end{figure}


\section{Examples}

As is well known in classical perspective drawing, as long as we can plot a grid of squares we can plot any object to any required precision, by caging it inside a fine enough grid and interpolating through the intermediate points. We will therefore concern ourselves with the basic examples of grid construction.

In fig. \ref{onepoint360} we build the image of a central uniform perspective grid. We consider a horizontal grid of squares (a tiled floor) with one axis parallel to $\overrightarrow{\spp{O}\spf}$ and the other parallel to $\overrightarrow{\spl\spr}$. For simplicity assume one of the grid's vertices is directly under the observer. Call ground plane to the plane of the grid and ground line to the intersection of the ground plane with the observer's plane. We make the plane of perspective represent also a top and a back orthogonal view of the scene. We make the back view of $\spp{O}$ coincide with $\ppf$, and scale the sphere to make it tangent to the ground plane at $\spd$. We make the top view of $\spp{O}$ coincide $\ppd$. In this way, a horizontal line through $\ppd$ represents both the ground plane on the back view and the observer's plane on top view. There is a grid line coincident with the ground line, and the receding lines of the grid intersect it at points $\spp{P_i}$ whose images $\orthoback{P_i}$ in back view are uniformly spaced. Since $\spp{P_i}$ is on the observer's plane, $P_i$ is obtained by intersecting ray $\overrightarrow{\ppf \orthoback{P_i}}$ with the equator by construction \ref{on_equator}. This ray, extended up to the blowup, is the perspective image of the central receding line of the grid that crosses $\spp{P_i}$. Thus the image of the receding lines of the grid is a set of radii $l_i$ going from $\ppf$ to the blowup, through the uniformly spaced $\orthoback{P_i}$. Note that this is analogous to the same construction in classical perspective, though with a different interpretation.

To plot the frontal lines of the grid we first trace a line $g$ on the ground plane, such that $g$ makes a 45 degree angle to the right of the observer and crosses $\ppd$. On top view we see that $g$ will diagonally cross a single square of each row of the grid. Hence it will touch each $l_i$ at a vertex of the grid. We plot the great circle $C$ of the plane defined by $O$ and $g$. First we plot the anterior half by drawing the arc $C_a=\ppd V \ppu$ where $V$ is the anterior  vanishing point of $g$, that lies on the $\ppl\ppr$ axis, 45 degrees to the right of $\ppf$. At each intersection of $C_a$ with an $l_i$ we mark a vertex of the grid, $G_i$,  and through it run a frontal line of the grid, drawing the arc of circle $\ppl G_i \ppr$.
For the $l_i$ that intersect $C$ on the posterior ring, intersect the antipodal line of $l_i$ (that is, the radius through $\antip{P_i}$) with $C_a$ to get a point $\antip{G_i}$, and take the antipode to find $G_i$, the vertex in the posterior ring. Draw the auxiliary frontal line $\ppr \antip{G_i}\ppl$, then construct its antipodal line $\ppr G_i \ppl$, using the $l_i$ as the natural measuring lines to draw its fat line approximation. This line $\ppr G_i \ppl$ is the frontal posterior grid axis through $G_i$. 
In this fashion we can plot the full 360 degree grid to any required precision and extension. 
Note that construction is analogous to that of a 1-point perspective grid in linear perspective, but here we get four vanishing points (counting the blowup circle  a single vanishing point), and we get six if we repeat the construction for the verticals (fig. \ref{onepoint360}, \ref{room45grid}).

In fig. \ref{room45grid} we represent a tiled cubic room drawn from the point of view of an observer at its center, looking straight into the center of one of the walls. The whole setup is drawn very simply from a judicious use of vertical and horizontal lines at 45 degrees to the observer; these lines do double duty, as, for instance, the vertical at 45 degrees to the right of the observer has the same great circle as the horizontal that goes under the observer at a 45 degree angle to his right. The same grid, with some further refinements, was used to draw the illustration on fig. \ref{room45pretty}. 

Often we will want our grids to be oriented at some arbitrary angle to the central axis.
In fig. \ref{arbitraryback} we represent a square $ABCE$ on a horizontal plane, below, behind, and to the left of the observer, such that one side of the square makes a 60 degree angle with $\overrightarrow{O\ppf}$. Once again the perspective plane also represents the top and back views of the scene, in the same setup as above. On the top view we draw the square $ABCE$ and project its sides until they intersect the top view of the observer's plane. We draw lines from $\ppf$ to these intersection points and find their projections on the equator. We find the vanishing points, all on the horizontal measuring line, one set of lines converging to the points at $60^\circ$ and $-120^\circ$ and the other to $-30^\circ$ and $150^\circ$. Through these points we find the arcs of circle corresponding to the lines that extend the sides of the square . From the arcs on the anterior perspective we obtain the corresponding fat lines of the posterior perspective. By intersecting these lines we find the perspective images of the points $A, B, C, E$. Finally, from this square we can plot a grid by an adaptation of the method already described.

\begin{figure}
\centering
\includegraphics[height=12cm]{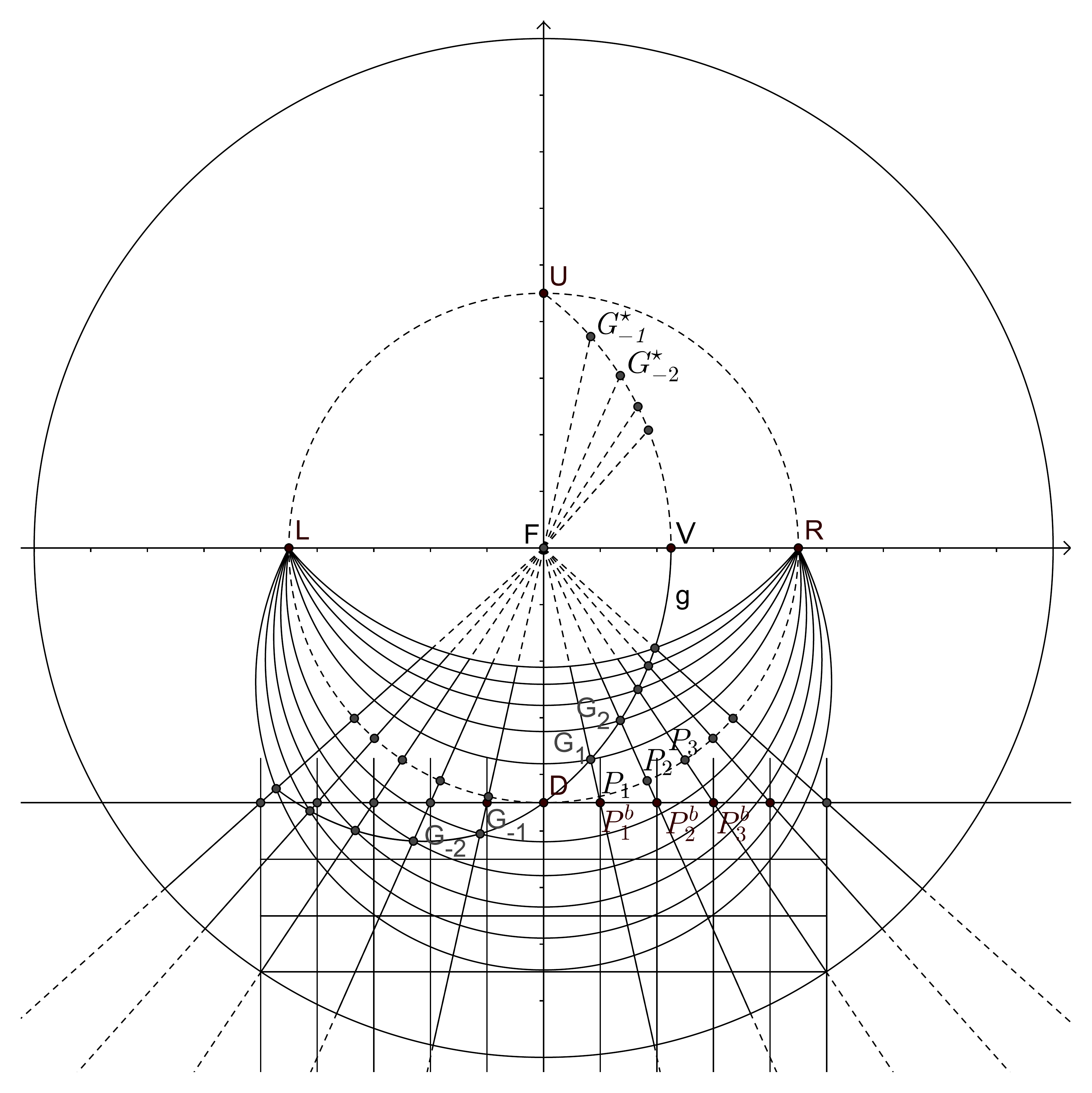}
\caption{Construction of a uniform central perspective grid. Lines converge to four vanishing points: $\ppl$, $\ppr$, $\ppf$, and $\ppb$.}\label{onepoint360}
\end{figure}

\begin{figure}
\centering
\includegraphics[height=8cm]{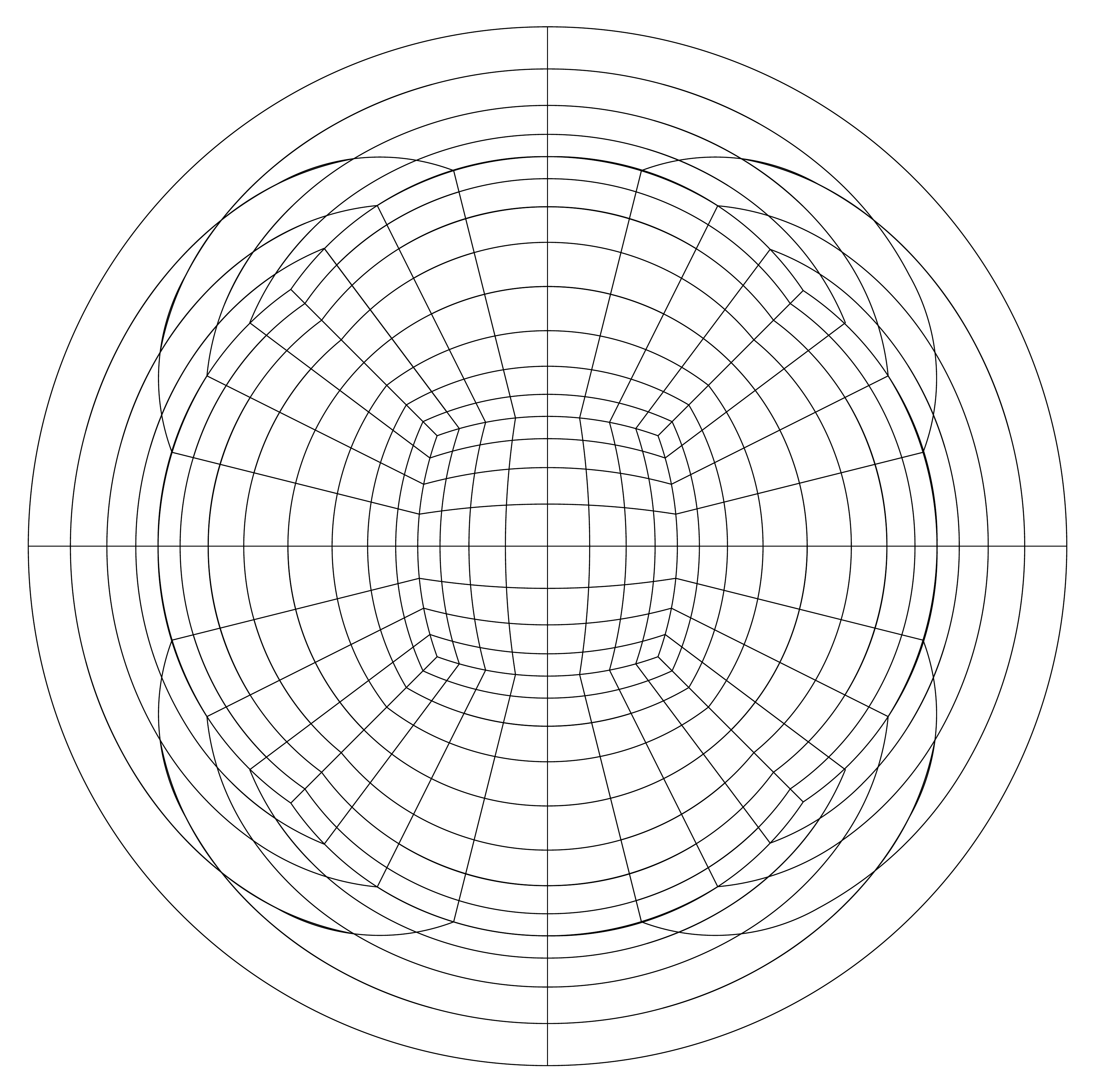}
\caption{A cubical box seen from its center. This is a six-point perspective construction (lines going to $\ppl$, $\ppr$, $\ppf$, $\ppb$, $\ppu$, $\ppd$). }\label{room45grid}
\end{figure}

\begin{figure}
\centering
\includegraphics[height=12cm]{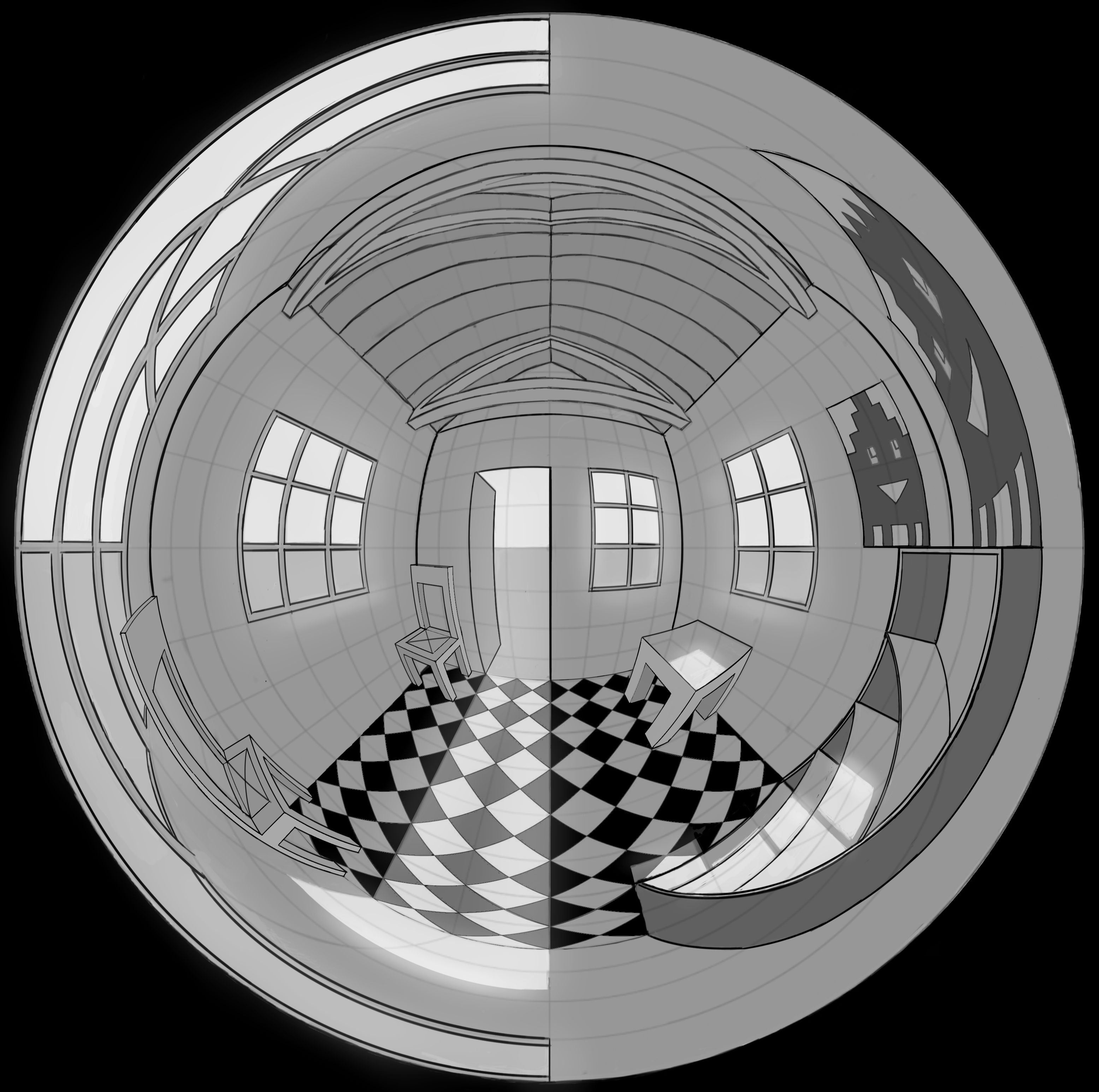}
\caption{Room 45. Drawing by the author of a cubical room using the construction of fig. \ref{room45grid}. The windows on the back and left walls have identical linear measurements, as do the pac-man figures on the right and back walls and the chairs on the front and back walls. This makes apparent the extent and nature of the deformations near the blowup. }\label{room45pretty}
\end{figure}

\begin{figure}
\centering
\includegraphics[height=11cm]{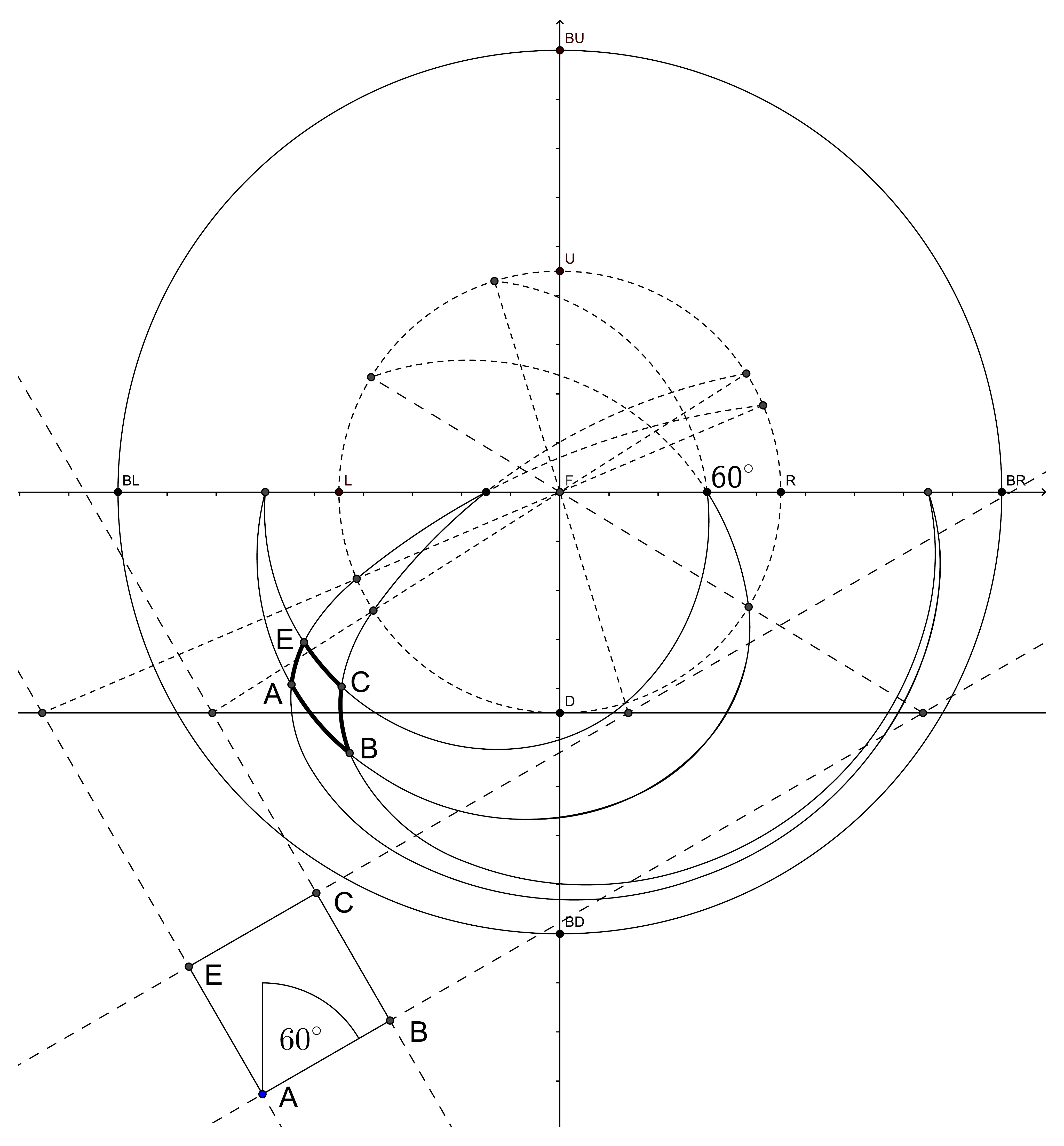}
\caption{A square, below, behind, and to the left of the observer.}\label{arbitraryback}
\end{figure}

\section{Comparison with reflections on a sphere}

It is apparent from the plot of the cubical room in fig \ref{room45pretty} that our perspective bears striking resemblances to a reflection on a sphere\footnote{An enterprising artist, more concerned with speed of execution than exactness might use it for an easy substitute of a true sphere reflection; the casual viewer might very well not notice the difference.}\cite{reflective}. It is natural to ask if there is a relation between the two. We shall now compare their properties.

Recall how reflection works (fig \ref{obstruction}): an observer at $E$ will see a  point $P$ reflected at a point $R$ on the sphere according to these rules: $R$ will be on the intersection of the sphere with the plane $EOP$ and $\angle (\overrightarrow{RP},\overrightarrow{OR})=\angle (\overrightarrow{RE},\overrightarrow{OR})$ (angle of incidence equals angle of reflection).

General reflections are hard to calculate. Given $R$, it is easy to find the incident and reflected rays, but the inverse problem of obtaining $R$ from $P$ is non-trivial. In general it requires solving an algebraic equation of order four (see \cite{REFLECTIONS}).

Also, occlusions are non-trivial. In fig. \ref{obstruction} we can see that points $P$ and $Q$ will have the same reflection $R$ even though they are not in the same ray from either the center $O$ or the observer $E$. This implies that a general reflection is not a central perspective. Recall that in central perspectives occlusions are always radial since they are determined at the anamorphosis step, whatever the flattening may be.

Finally, total spherical perspective has an angle of view of $360^\circ$, while 
the angle of view captured by a reflection depends on the distance of the observer to the sphere. The points of the sphere define a cone with the observer $E$ at the vertex, the cone of shadow, and every point outside of this cone of shadow will be viewable on the sphere.  The field of view will be $360^\circ-\delta$ with $\delta=2\sin^{-1}(r/d)$, where $r$ is the radius of the sphere and $d$ the distance of the observer from the center of the sphere. 

There is however a limiting case where spherical perspective and reflection on a sphere become quite similar.
Imagine either moving away from the sphere (preserving its apparent size by looking at it through a telescope) or shrinking it (and seeing at it through a microscope). Then $r$ becomes small compared to $d$ and, in the limit $r/d \to 0$, we get a $360^\circ$ angle of view. $ER$ becomes parallel to $EO$, the angle of reflexion $\alpha$ becomes equal to $\beta=\angle EOR$, and $\angle ERP\to 2 \alpha$ (see fig. \ref{parallelreflection}).
If furthermore $r \to 0$ (an infinitesimal sphere) or $r/|OP|\to 0$ while $\lambda=\angle EOP$ remains constant (reflection  of points on the celestial sphere) then $\lambda\to 2 \alpha$. 
In this limit, the projection becomes radial (therefore making occlusions trivial), and the whole space of directions is mapped onto the hemisphere visible from $E$. This can be seen as a sphere anamorphosis followed by a uniform contraction onto a hemisphere by halving the angle $\angle EOR$ of each point $R$ of the sphere.

Seen from point $E$, since all rays $ER$ are parallel to the axis $OE$, the reflection will look like the  orthogonal projection along $OE$ of the image on the sphere. Hence the reflection, seen from $E$, is anamorphically equivalent to a central perspective (central with respect to $O$, not $E$) obtained by anamorphosis onto the sphere followed by a flattening which is the composition of a uniform compression onto a hemisphere followed by an orthogonal projection. In the spherical coordinates of equation \ref{natcoords} (with the $y$ axis on $\overrightarrow{OE}$ and $x,z$ in the perpendicular plane through $O$) and rescaling the sphere to $r=1$, this perspective is the map 
\begin{equation*}
(\rho,\lambda,\theta) \mapsto (1,\lambda,\theta) \mapsto (1,\lambda/2,\theta)\mapsto \sin (\lambda/2) (cos(\theta),sin(\theta))
\end{equation*}
where the first map is the anamorphosis, the second is the crunching into the anterior hemisphere and the last step is the orthogonal projection onto the disc perpendicular to $EO$ at $O$.

This is a 360 degree perspective, but different from our spherical perspective. It is not linear along $\lambda$, squashing the outer angles more, and cannot be easily used for drawing by hand  without the help of pre-computed grids (since we lose the isometry along measuring lines). But we can see why there is a qualitative similarity between the two.

\begin{figure}[H]
\centering
\parbox{5.5cm}{
\includegraphics[width=5.5cm]{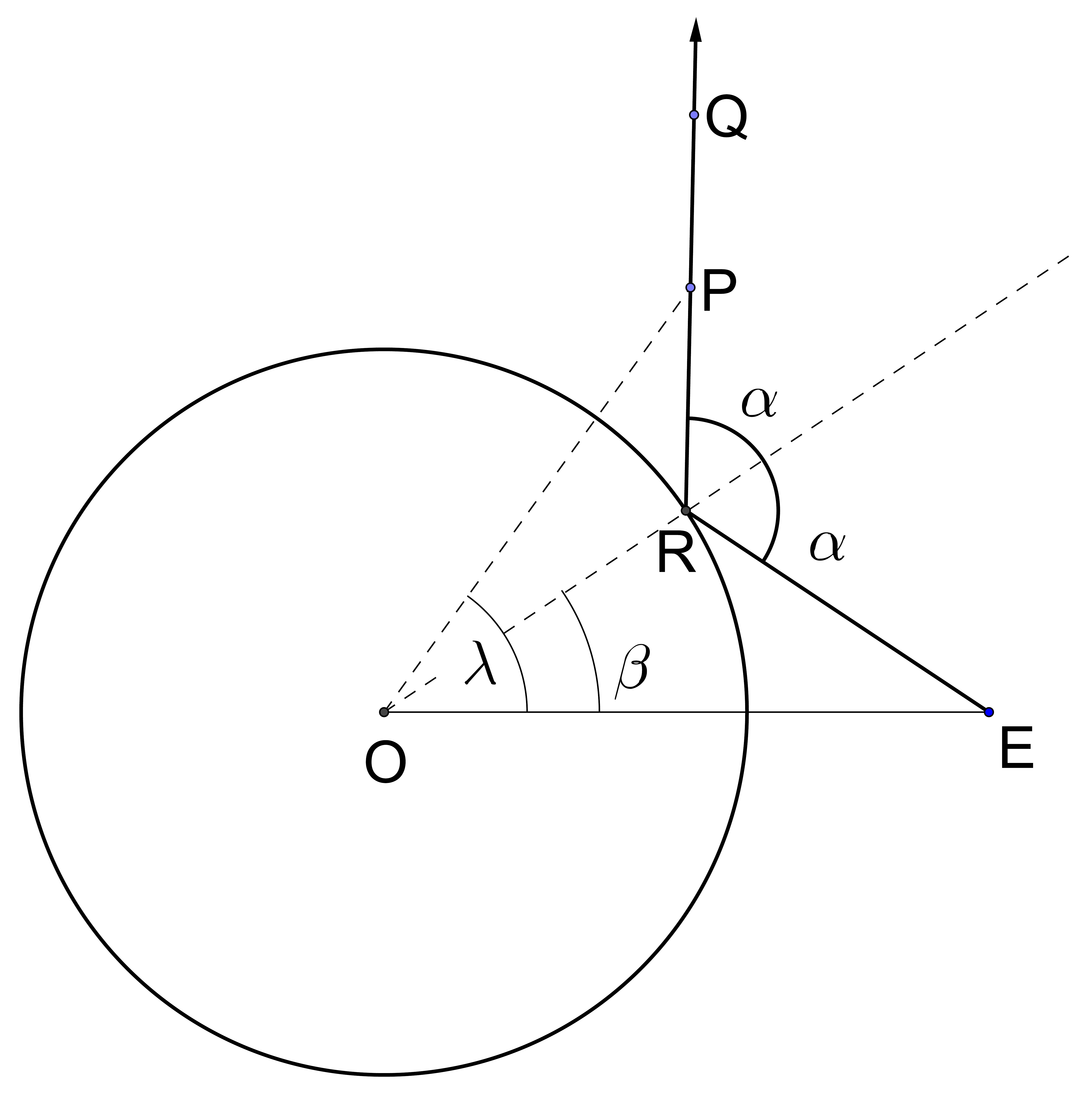}
\caption{Non-radial occlusion. 
\\
Points P and Q both project to R although they are not in the same ray from E or O.}
\label{obstruction}}
\qquad
\begin{minipage}{6cm}
\includegraphics[width=7cm]{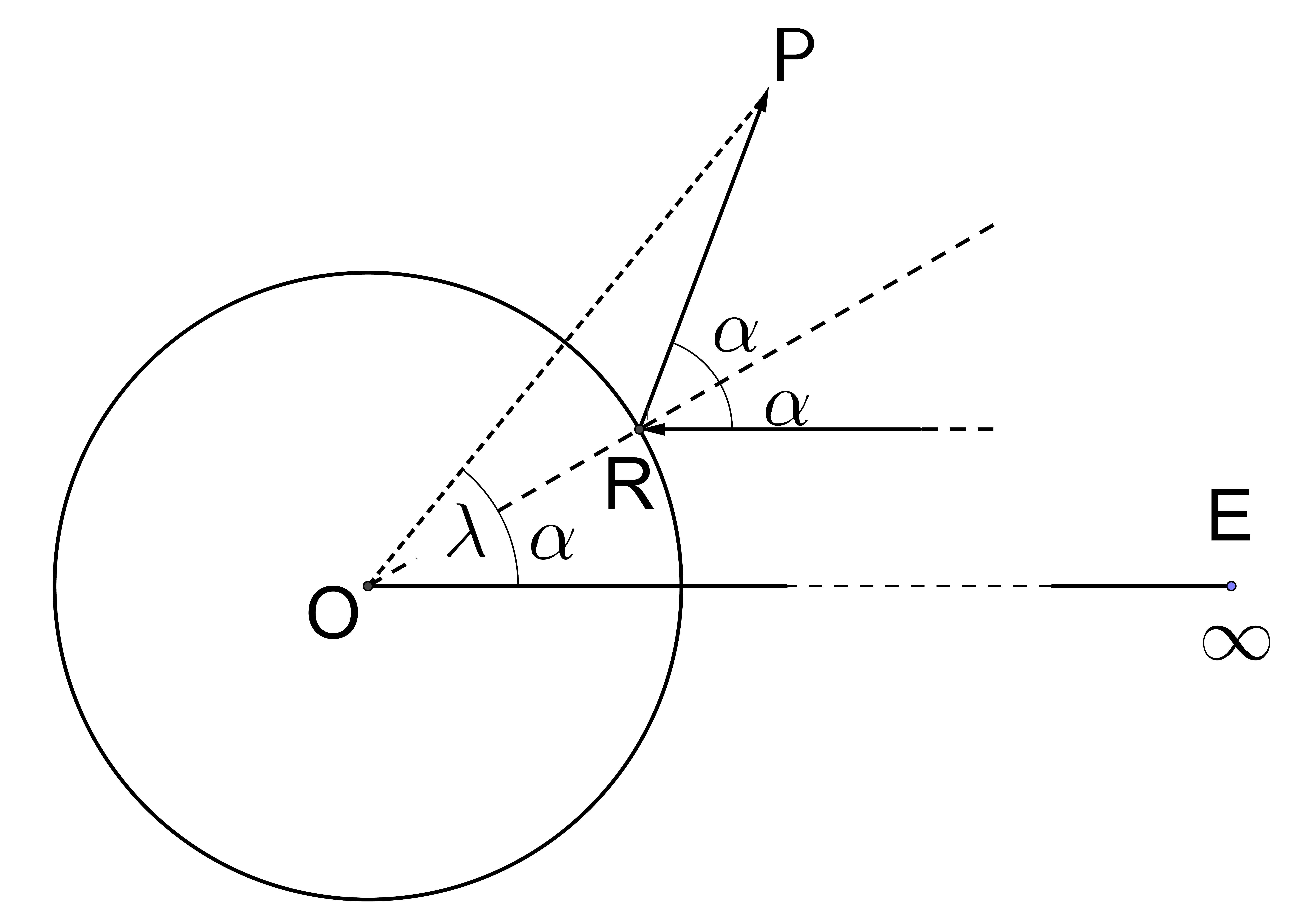}
\caption{With E at infinity all rays become parallel and angle $\beta$ becomes equal to $\alpha$. If P goes to infinity (with fixed $\lambda$) then $\alpha$ goes to $\lambda/2$.}
\label{parallelreflection}
\end{minipage}
\end{figure}

It has been noted in \cite{REFLECTIONS} that reflections on a sphere could be used as a form of wide angle perspective. This is well inspired in art history, as reflections drawn from observation have been the time honoured tool of the artist to represent a wide angle of view, Escher's self portrait being a well known example. 

But we have seen the difficulties in this approach. First, reflections are hard to calculate. Second, they are not central perspectives, and they have non-trivial occlusions. As noted in  \cite{REFLECTIONS} this causes difficulties for hidden-face removal algorithms. Even in the limit presented above, where it becomes a central perspective, it is clear that a sphere reflection only makes for a practical perspective for the draughtsman when drawn from observation of an actual sphere.

Spherical perspective is a much more natural proposal for a wide perspective. It allows for up to a 360 degree view, it is easily computed by equation \ref{exactxyz}, it is a central perspective and therefore has trivial occlusions, so   hidden-face algorithms will work exactly as in the classical case, being calculated at the anamorphosis step. Most important for our purposes, it lends itself to be used by an artist armed only with his ruler, compass, and eventual nail. With some practice even these instruments can be abandoned in favor of reasonably intuitive and accurate freehand drawing from either nature or the imagination.

%
%
Note: Further notes, computer code and illustrations will be made available at the author's web page: 
\url{http://www.univ-ab.pt/~aaraujo/full360.html}
\section{Acknowledgements}
This work was supported by Funda\c{c}\~ao para a Ci\^encia e a Tecnologia (FCT) project UID/MAT/04561/2013.

\,
\,



\end{document}